\documentclass[a4paper,11pt,oneside]{article}
\usepackage{graphicx}
\usepackage{amscd}
\usepackage{amsmath}
\usepackage{caption}
\usepackage{amsfonts}
\usepackage{amssymb}
\usepackage{amsthm}
\usepackage{mathrsfs}
\usepackage{multicol}
\usepackage{color}
\usepackage[english]{babel}
\usepackage[T1]{fontenc}
\usepackage{textcomp}
\usepackage[utf8]{inputenc}
\usepackage{enumitem}
\usepackage{indentfirst}
\usepackage{tikz}
\usetikzlibrary{positioning, fit, arrows.meta, shapes, matrix}
\usepackage{float}
\usepackage{comment}

\newcommand{\N}{\mathbb N}

\newcommand{\RR}{{{\rm I} \kern -.15em {\rm R} }}

\DeclareMathOperator\supp{supp}

\begin{document}
\theoremstyle{plain} 
\newtheorem{thm}{Theorem}[section] 
\newtheorem{cor}{Corollary} [section]
\newtheorem{lem}{Lemma}[section]
\newtheorem{prop}{Proposition}[section]
\theoremstyle{definition} 
\newtheorem{defn}{Definition}[section] 
	
	\newtheorem{oss}{Remark}[section]
	\newtheorem{ex}{Example}[section]
	\newtheorem{lemma}{Lemma}[section]
	
	\title{Flocking and Mean-Field Analysis of Delayed Leader-Follower Cucker-Smale System}
	\author{Chiara Cicolani \footnote{ Email: cicolani@uni-wuppertal.de} \\School of Mathematics and Natural Sciences,\\
		University of Wuppertal, Germany}
	\date{}

	\maketitle

	\begin{abstract}
Inspired by \cite{CCP}, we investigate a delayed leader-follower Cucker-Smale model describing the collective dynamics of interacting agents subject to communication lags. We first study the particle dynamics and establish sufficient conditions ensuring the emergence of asymptotic flocking. Our analysis shows that velocity alignment and bounded spatial dispersion persist despite the presence of delays and heterogeneous interactions between leaders and followers. We then derive and analyze two continuum descriptions of the system. In the first regime, the number of leaders is kept fixed while the number of followers tends to infinity, leading to a hybrid particle-kinetic model. In the second regime, both populations become infinitely large, yielding a fully kinetic delayed leader-follower model. For both mean-field formulations, we prove global existence, uniqueness, and Wasserstein stability of measure-valued solutions. These results provide a rigorous mathematical framework for the study of collective dynamics with leadership and memory effects and establish a bridge between delayed flocking models and their continuum counterparts.
\end{abstract}

	\providecommand{\keywords}[1]{\textbf{Keywords:} #1}
	
	\vspace{5 mm}
	
	\section{Introduction}\label{introduction}

	Collective motion in multi-agent systems arises in biology \cite{Cama, Carrillo, CS1}, physics \cite{DH, Stro}, and robotics \cite{Bullo, Desai}, among other fields  \cite{Almeida2, Bellomo, Castellano, D, Dolfin, Marsan, PRT, Sole}. In this context, the Cucker-Smale model \cite{CS1} was introduced to describe the emergence of flocking through interactions among agents with positions and velocities, and in recent years, several generalizations of this model have been proposed by many authors (see e.g. \cite{Cic_Con_Pign, ContPign, DH, Ha1, Ha2, Haskovec, LW, PRT, PR, PT}). A natural extension accounts for delays in communication and information processing: an agent may respond to the state of another agent observed at an earlier time, due to the reaction time or the time necessary to propagate information.  The presence of a delay in the interaction makes the problem more delicate to analyze. The standard instantaneous dissipation estimates for the velocity fluctuations may no longer apply directly, and establishing velocity alignment therefore requires estimates that also control the agents' past states. For well-posedness results for alignment models in the presence of time delay effects, we refer to classical texts on functional differential equations \cite{Halanay, Hale}.  \\
Several aspects of delayed flocking models have been studied separately in the literature. For example, point delays may be constant or time-varying and can differ across agents. A further distinction concerns how delays enter the interaction: communication-delay models combine a neighbour’s past velocity with the agent’s current velocity, whereas reaction-delay models may evaluate both velocities in the past. \\
For point delays, particle flocking has been studied with normalized weights \cite{CH}, heterogeneous delays \cite{CL}, and hierarchical leadership \cite{PR}. Pignotti and Trélat \cite{PT} studied time-varying delays with symmetric and nonsymmetric communication weights under suitable restrictions on the delay. Recently, Rodriguez Cartabia
\cite{Cartabia} showed that the unconditional-flocking criterion of the undelayed system remains valid for any delay size. Subsequently, Continelli \cite{Cont} established exponential flocking for bounded time-varying delays
without requiring their upper bound to be small. \\
In the present work, we consider distributed delays that describe a weighted average of observations over a past time interval. At the particle level, distributed-delay flocking was further studied with general communication weights in \cite{Liu20} and with reaction delays in \cite{HM20}. For normalized communication weights, Choi and Pignotti \cite{CP} derived particle flocking conditions and a delayed Vlasov limit with measure-valued solutions. In \cite{ZYG23}, the authors obtained both flocking and non-flocking criteria for distributed-delay systems with short-range interactions. More recently, Continelli \cite{Continelli26} considered communication failures and time-dependent delays, obtained a decay rate independent of the number of agents for universal interactions, and extended the analysis to distributed delays. \\
Leadership and delayed flocking have also been considered together: Pignotti and Reche Vallejo treated hierarchical leadership with a point delay in \cite{PR} and with a distributed delay in \cite{PRV19}. Other results for non-universal interactions and intermittent communication on strongly connected digraphs were obtained in \cite{Cic_Con_Pign} at the particle level. \\
 At the kinetic level, delayed mean-field limits for a single population of Cucker-Smale agents were established in \cite{CH} for a pointwise delay and in \cite{CP}  for distributed delays, in both cases with normalized communication weights. A different delay mechanism was studied in \cite{Haskovec}: finite-speed information propagation leads to state-dependent delays and a mean-field formulation on trajectory space. \\
 A particularly close two-population comparison for this paper is \cite{CCP}:  for a first-order opinion model with constant delays, the authors proved consensus and studied both a hybrid limit with finitely many leaders and a full mean-field limit. In a second-order system, the kinetic variables include positions and velocities, and velocity alignment alone does not control the growth of the position diameter. Therefore, it is more difficult to analyze. \\
We study a delayed Cucker-Smale system with two interacting populations. Leaders interact among themselves, whereas followers respond both to other followers and to leaders. The interactions involve three possibly different spatial influence functions and a memory kernel, and the
lengths of the leader and follower memory windows may vary in time.\\
Consider two finite sets of $N,m \in\N$ particles, with $N,m\geq 2 $. We consider \(m\) leaders and \(N\) followers. Let $(y_{j}(t), w_j(t))_{j=1}^m \in (\RR^{2d})^m$ be the position and the velocity of the $j$-th leader at time $t.$ Analogously, let us identify the position and the velocity of the $i$-th follower at time $t $ with $(x_{i}(t), v_i(t))_{i=1}^N \in (\RR^{2d})^N.$ We shall denote with $\lvert\cdot \rvert$ and $\langle \cdot,\cdot \rangle$ the usual norm and scalar product in $\RR^{d}.$ 
The interactions between the elements of the system are described by the following equations,
	\begin{equation}\label{leaders}
		\begin{cases}
			\frac{d}{dt}y_{i}(t)=w_{i}(t),\quad &t>0, \,\,\forall i=1,\dots,m,\\
			\frac{d}{dt}w_{i}(t)=\frac{1}{m}\sum_{j=1}^m \int_{t-\tau_1(t)}^{t}  G(t-s)c_{ij}(t,s)(w_{j}(s)-w_{i}(t)) ds,\quad 	&t>0,\,\,\forall i=1,\dots,m, 
		\end{cases}
	\end{equation}
	
	\begin{equation}\label{followers}
		\begin{cases}
			\frac{d}{dt}x_{i}(t)=v_{i}(t),\quad &t>0, \,\,\forall i=1,\dots,N,\\
			\frac{d}{dt}v_{i}(t)=\frac{1}{N}\sum_{j=1}^N \int_{t-\tau_2(t)}^{t} G(t-s)a_{ij}(t,s)(v_{j}(s)-v_{i}(t)) ds \\
			\hspace{1.5cm}+\frac{1}{m}\sum_{j=1}^m \int_{t-\tau_1(t)}^{t}  G(t-s)b_{ij}(t,s)(w_{j}(s)-v_{i}(t)) ds,\quad 	&t>0,\,\,\forall i=1,\dots,N.
		\end{cases}
	\end{equation}
In this formulation, the length of the memory can change, according to two time delay functions $\tau_1(t),\tau_2(t) \in C([0,+\infty);(0,+\infty)),$ corresponding to the reaction time of the leaders and the reaction time of the followers, respectively. We assume that the time delay functions have a lower and an upper bound, i.e., there exist two positive constants $\tau^*$ and $\bar\tau,$ such that 
\begin{equation}\label{delay_ass}
0< †\tau^* \leq \tau_1(t),\tau_2(t) \leq \bar\tau, \quad \forall \ t \geq 0.
\end{equation}
Moreover, we consider a function $G:[0,+\infty) \to [0,+\infty),$ $ G \in \text{L}^1_{\text{loc}}([0,+\infty)),$ which represents the memory kernel and satisfies the following conditions: 
\begin{itemize}
\item[(A1)] 
\begin{equation}\label{memory_cond}
\int_{0}^{\bar \tau}G(s) ds \leq 1; 
\end{equation} 
\item[(A2)]  for $i=1,2,$ considering the quantities 
$$ g_i(t):=\int_{t-\tau_i(t)}^t G(t-s) \, ds, $$ 
there exists a positive constant $\eta$ such that 
\begin{equation}\label{PE_memory}
\int_{t}^{t+ \bar \tau }g_1(s) ds \geq \eta, \quad \forall t \ge 0.
\end{equation}
\end{itemize}
The integral condition on the memory kernel can be interpreted as a persistent excitation assumption on its effective mass, ensuring a uniformly positive cumulative interaction strength over every time interval of a fixed length (see \cite{Continelli26}).\\
The communication rates are defined by 
	\begin{equation}\label{weightcs}
	\begin{split}
	& a_{ij}(t,s):=\phi(\lvert x_{i}(t)-x_{j}(s)\rvert), \quad t>0, s \in [t-\tau_2(t),t], \quad i,j=1,\dots,N, \\
	& b_{ij}(t,s):=\rho(\lvert x_{i}(t)-y_{j}(s)\rvert), \quad t>0,  s \in [t-\tau_1(t),t], \quad i=1,\dots,N, \ j=1,\dots,m \\
		& c_{ij}(t,s):=\psi(\lvert y_{i}(t)-y_{j}(s)\rvert), \quad t>0, s \in [t-\tau_1(t),t], \quad i,j=1,\dots,m, \\
		\end{split}
	\end{equation}
	where $\psi,\phi,\rho :[0, +\infty)\rightarrow \RR$ are positive, uniformly bounded, and Lipschitz continuous functions, and we call  
\begin{equation} \label{K}
 K:=\max \left\{ \lVert  \psi\rVert_{\infty}, \lVert \phi\rVert_{\infty}, \lVert \rho\rVert_{\infty}\right\}.
\end{equation}
The initial conditions
	\begin{equation}\label{incondcs_leaders}
		y_{i}(s)=y^{0}_{i}(s),\quad w_{i}(s)=w^{0}_{i}(s), \quad s\in [-\bar\tau,0],\, i=1,\dots,m,
	\end{equation}
	and 
	\begin{equation}\label{incondcs_followers}
		x_{i}(s)=x^{0}_{i}(s),\quad v_{i}(s)=v^{0}_{i}(s), \quad s\in [-\bar\tau,0],\, i=1,\dots,N,
	\end{equation}
 are assumed to be continuous functions. We set
	\begin{equation}\label{C0}
		C^V_{0}:=\max_{s \in [-\bar\tau,0]} \left\{\max_{i=1,\dots,m}\,\,\lvert w_{i}(s)\rvert,\max_{i=1,\dots,N}\,\,\lvert v_{i}(s)\rvert\right\},
	\end{equation} 
	\begin{equation}\label{MX}
	\begin{split}
		&M_0^{X}:=\max_{s,t \in [-\bar\tau,0]} \Big\{\max_{i=1,\dots,m}\lvert y_{i}(s)-y_i(t)\rvert,\max_{i=1,\dots,N}\lvert x_{i}(s)-x_i(t)\rvert, \\
		& \hspace{7 cm} \max_{i=1,\dots,m}\,\, \max_{j=1,\dots,N}\lvert y_{i}(s)-x_j(t)\rvert\Big\}.
		\end{split}
	\end{equation}
	We define the space and velocity diameters as follows 
	\begin{equation}\label{diamX}
	\begin{split}
	& d_{X}(t):=\max\Big\{\max_{i,j=1,\dots,m}\lvert y_{i}(t)-y_{j}(t)\rvert, \max_{i,j=1,\dots,N}\lvert x_{i}(t)-x_{j}(t)\rvert, \\
	& \hspace{7 cm} \max_{i=1,\dots,m} \max_{j=1,\dots,N}\lvert y_{i}(t)-x_{j}(t)\rvert\Big\},
	\end{split}
	\end{equation}
	
	\begin{equation}\label{diamV}
	\begin{split}
 & d_{V}(t):=\max \Big\{\max_{i,j=1,\dots,m}\lvert w_{i}(t)-w_{j}(t)\rvert, \max_{i,j=1,\dots,N}\lvert v_{i}(t)-v_{j}(t)\rvert, \\
 & \hspace{7cm} \max_{i=1,\dots,m}\max_{j=1,\dots,N}\lvert w_{i}(t)-v_{j}(t)\rvert\Big\},
 \end{split}
	\end{equation}
	
		for all $t\geq -\bar\tau.$  Moreover, we define the following quantities.
\begin{defn}\label{Dn}
For all $n\in \mathbb{N}_0$, we define the \emph{local velocity diameters} as 
	$$D_n:=\max_{r,s\in [n\bar\tau-\bar\tau,n\bar\tau]}\Big\{\max_{i,j =1,\dots, m}\lvert w_i(r)-w_j(s)\rvert, \, \max_{i,j =1,\dots, N}\lvert v_i(r)-v_j(s)\rvert, \, \max_{i=1,\dots, m} \max_{j=1,\dots,N} \lvert w_i(r)-v_j(s)\rvert \Big\}.$$
\end{defn}
In particular, for $n=0$, 
$$D_0:=\max_{r,s\in [-\bar\tau,0]}\left\{ \max_{i,j =1,\dots, m}\lvert w_i(r)-w_j(s)\rvert, \, \max_{i,j =1,\dots, N}\lvert v_i(r)-v_j(s)\rvert, \, \max_{i=1,\dots, m} \max_{j=1,\dots,N} \lvert w_i(r)-v_j(s)\rvert \right\}.$$

	\begin{defn} \label{unflock} 
We say that a solution $\{(y_i,w_i)_{i=1}^m,(x_{j},v_{j})_{j=1}^N\}$ to system \eqref{leaders}-\eqref{followers} exhibits \textit{asymptotic flocking} if it satisfies the following two conditions:
		\begin{enumerate}
			\item there exists a positive constant $d^{*}$ such that$$\sup_{t\geq-{\bar\tau}}d_{X}(t)\leq d^{*};$$
			\item$\underset{t \to\infty}{\lim}d_{V}(t)=0.$
		\end{enumerate}
	\end{defn}
Our first result gives bounded spatial dispersion and exponential alignment of the entire population under a joint divergent-tail condition. The quantitative step specific to the two-population system is a contraction for the velocity diameter $D_n$ measured over successive memory windows. The flocking criterion does not impose a smallness restriction on the bounded delay lengths. 
	\begin{thm} \label{uf}
Let $\psi, \phi, \rho: [0, +\infty)\rightarrow\RR$ be positive, bounded, and Lipschitz continuous functions that satisfy
	\begin{equation}\label{infint}
    \int_{0}^{+\infty}\kappa(r)\,dr=+\infty,
    \qquad
    \kappa(r):=
    \min_{\sigma\in[0,r]} \{\psi(\sigma),\phi(\sigma),\rho(\sigma)\}.
\end{equation}
 Moreover, let $x^{0}_{i},v_{i}^{0}:[-{\bar\tau},0]\rightarrow \RR^{d}$ and $y_j^0, w_j^0:[-\bar\tau,0]\rightarrow \RR^d$ be continuous functions for any $i=1,\dots,N$ and $j=1,\dots,m.$ Assume that the delays $\tau_1,\tau_2$
and the nonnegative memory kernel $G$ satisfy the delay and
memory-kernel hypotheses given by \eqref{delay_ass}, \eqref{memory_cond} and \eqref{PE_memory}. 
Then, for every solution $\{(y_j, w_j)\}_{j=1,\dots,m}$and $\{(x_{i},v_{i})\}_{i=1,\dots,N}$ to \eqref{leaders} and \eqref{followers} with the initial conditions \eqref{incondcs_leaders} and \eqref{incondcs_followers} respectively, there exists a positive constant $d^{*}$ such that \begin{equation}\label{posbound}
			\sup_{t\geq-{\bar\tau}}d_{X}(t)\leq d^{*},
		\end{equation}
		and the following exponential decay estimate holds
		\begin{equation}\label{vel}
			d_{V}(t)\leq C e^{-\gamma(t-2\bar\tau)},\quad \forall t\geq 0,
		\end{equation}
		where $C$ and $\gamma$ are suitable positive constants, and $\bar\tau$ the size of the time delay.
 \end{thm}

	\begin{oss} \label{condizioneintegrale}
Condition \eqref{infint} is a sufficient assumption for unconditional flocking in the present leader-follower system. The function  $\kappa$ is a common nonincreasing lower envelope of the three communication functions and provides a lower bound
for each of them on $[0,r]$. If $\psi=\phi=\rho=\Psi$, with $\Psi$ nonincreasing, then \eqref{infint} reduces to the classical condition
$$ \int_{0}^{+\infty}\Psi(r)\,dr=+\infty, $$
used, for instance, in \cite{Cartabia}. In general, \eqref{infint} is more restrictive than requiring integral divergence for each communication function separately. We adopt this assumption to streamline the analysis and provide a unified treatment, but the same argument applies to both monotone and nonmonotone communication functions satisfying condition \eqref{infint}. A similar lower-envelope assumption for a single, not necessarily monotone communication function is used in \cite{Cont}.
\end{oss}

To understand the collective behavior of a large number of interacting agents, we next study the mean-field limit of the delayed particle system introduced by the models \eqref{leaders} and \eqref{followers}. The goal of this analysis is twofold: first, to derive macroscopic equations that describe the evolution of the system when the number of agents becomes large; and second, to establish flocking results at the mean-field level. For the resulting measure-valued systems, we establish existence and uniqueness for compactly supported initial histories, Wasserstein stability on every bounded time interval, and the inheritance of the flocking estimates. The characteristic-flow construction and transport-based stability argument adapt methods used for one-population kinetic alignment and delayed mean-field equations \cite{Carrillo, CH, CP, Ha1}. The estimates that require new treatment here are the cross-population velocity contraction and the stability bounds for the coupled leader-follower forces with two time-dependent memory windows. Our stability constants depend on the time horizon; uniform-in-time mean-field stability is a stronger property studied for the undelayed
Cucker-Smale system \cite{HaKimZhang18,HaSurvey25}. \\

We investigate two distinct asymptotic regimes that reflect different population structures.  \\

\noindent{\it Case (i): few leaders and many followers.} In this regime, the number of leaders $m$ remains fixed, while the number of followers $N$ tends to infinity. This setup reflects situations in which a small number of individuals shape the dynamics of a much larger population. The leaders retain their finite-dimensional dynamics, according to \eqref{leaders}, whereas the followers are described by a probability measure $f_t$ governed by a Vlasov-type kinetic equation of the form 
\begin{equation}\label{pde1}
		\partial_t f_t+ v \cdot \nabla_x f_t+ \nabla_v \cdot \left(F^m[f](t,x,v)f_t\right)=0, \quad x,v \in \mathbb{R}^d, \ t >0,
	\end{equation}
where $f:t \mapsto f_t$ denotes the measure-valued curve.
The initial histories are
$$
(y_i(s),w_i(s))=(y_i^0(s),w_i^0(s)),
\qquad f_s=\nu_s,
\qquad s\in[-\bar\tau,0],\quad i=1,\ldots,m.
$$
The mean-field interaction potential for the followers' measures is given by
\begin{equation}\label{flux1}
\begin{split}
& F^m[f](t,x,v) := \int_{t-\tau_2(t)}^t G(t-s) \int_{\mathbb{R}^{2d}} \phi(|x-\bar{x}|)(\bar{v}-v)f_{s}(d\bar{x},d\bar{v}) \ ds \\
& \hspace{5cm} +  \frac{1}{m} \sum_{j=1}^{m} \int_{t-\tau_1(t)}^t G(t-s) \rho(|x- y_j(s)|) \bigl( w_j(s)-v\bigr) ds.
\end{split}
\end{equation}

\noindent{\it Case (ii): infinite population limit for both leaders and followers.} In this fully macroscopic regime, both groups are described by probability measures. The mean-field limit then yields a pair of kinetic equations:
\begin{equation}\label{pde2}
\begin{array}{l}
\displaystyle{
 \partial_t g_t + w \cdot \nabla_y g_t +\nabla_w \cdot \left( F_L[g](t,y,w)g_t\right)= 0, \quad y,w \in \mathbb{R}^d, \quad t>0,} \\[2mm]
\displaystyle{
 \partial_t f_t + v \cdot \nabla_x f_t + \nabla_v \cdot (F_F[f,g](t,x,v)f_t) = 0, \quad x,v \in \mathbb{R}^d, \quad t>0.}
\end{array}
\end{equation}
As above, $f:t \mapsto f_t$ and $g: t \mapsto g_t$ denote the measure-valued curves.
Here, the  mean-field interaction potentials are defined as
\begin{align}\label{flux2}
F_L[g](t,y,w)&:= \int_{t-\tau_1(t)}^t G(t-s)\int_{\mathbb{R}^{2d}} \psi(|y-\bar{y}|)(\bar{w}-w)\,g_{s}(d \bar{y},d \bar{w}) \ ds,  \\ \label{flux3}
F_F[f,g](t,x,v)&:= \int_{t-\tau_2(t)}^t G(t-s) \int_{\mathbb{R}^{2d}} \phi(|x-\bar{x}|)(\bar{v}-v)\,f_{s}(d \bar{x}, d \bar{v}) \ ds \cr
& \hspace{1.5cm}+ \int_{t-\tau_1(t)}^t G(t-s) \int_{\mathbb{R}^{2d}} \rho(|x-\bar{y}|)(\bar{w}-v)\,g_{s}(d \bar{y}, d \bar{w}) \ ds. 
\end{align}
To study velocity alignment and spatial dispersion in these mean-field models, we introduce diameter-like quantities that measure the spread of the distributions. We refer to $\operatorname{supp}_X(\cdot)$ and $\operatorname{supp}_V(\cdot)$ as the support of the measure with respect to space and velocity, respectively. \\
In Case (i), we define:
\[%\begin{equation}\label{d_nu}
\begin{split}
& d_X^{f}(t):= \max \left\{ \sup_{x,y \in \operatorname{supp}_X(f_t)}|x-y|, \; \max_{i,j=1,\dots,m}|y_i(t)-y_j(t)|, \;  \max_{i=1,\dots,m}\sup_{x \in \operatorname{supp}_X(f_t)}|y_i(t)- x| \right\}, \\
& d_V^{f}(t):= \max \left\{ \sup_{v,w \in \operatorname{supp}_V(f_t)}|v-w|, \; \max_{i,j=1,\dots,m}|w_i(t)-w_j(t)|, \;  \max_{i=1,\dots,m}\sup_{v \in \operatorname{supp}_V(f_t)}|w_i(t)- v| \right\}.
\end{split}
\]%\end{equation}
In Case (ii), the diameters become: 
\[%\begin{equation}\label{d_mu_nu}
\begin{split}
& d_X^{f, g}(t):= \max \left\{ \sup_{x,y \in \operatorname{supp}_X(g_t)}|x-y|, \;  \sup_{x,y \in \operatorname{supp}_X(f_t)}|x-y|,  \; \sup_{\substack{x \in \operatorname{supp}_X(g_t),\\ y \in \operatorname{supp}_X(f_t)}}|x-y| \right\}, \\
& d_V^{f, g}(t):= \max \left\{ \sup_{v,w \in \operatorname{supp}_V(g_t)}|v-w|, \;  \sup_{v,w \in \operatorname{supp}_V(f_t)}|v-w|,  \; \sup_{\substack{v \in \operatorname{supp}_V(g_t),\\ w \in \operatorname{supp}_V(f_t)}}|v-w| \right\}.
\end{split}
\]%\end{equation}
To measure initial velocity discrepancies, we define:
\[%\begin{equation} \label{gen_diam}
\begin{split}
D^f_0 &:= \sup_{s,t \in [-\bar\tau,0]}  \left\{ \sup_{ \substack{v \in \operatorname{supp}_V(\nu_s),\\ w \in \operatorname{supp}_V(\nu_t)}}|v-w|, \;  \max_{i,j=1,\dots,m}|w_i(s)-w_j(t)|, \;\max_{i=1,\dots,m}\sup_{v \in \operatorname{supp}_V(\nu_t)}|w_i(s)- v| \right\}, \\
D^{f, g}_0 &:= \sup_{s,t \in [-\bar\tau,0]}  \left\{ \sup_{\substack{v \in \operatorname{supp}_V(\bar \mu_s),\\ w \in \operatorname{supp}_V(\bar \mu_t)}}|v-w|, \; \sup_{\substack{v \in \operatorname{supp}_V(\bar \nu_s),\\ w \in \operatorname{supp}_V(\bar \nu_t)}}|v-w|,  \; \sup_{ \substack{v \in \operatorname{supp}_V(\bar \mu_s),\\  w \in \operatorname{supp}_V(\bar \nu_t)}}|v-w| \right\}.
\end{split}
\]%\end{equation}
\medskip

We use standard notation for probability measures and transport; see \cite{Carrillo}.
Let $\mathcal{P}_1(\mathbb{R}^{2d})$ denote the set of Borel
probability measures with finite first moment, endowed with
the $1$-Wasserstein distance
\[
d_1(\mu,\nu)
:=
\inf_{\pi\in\Pi(\mu,\nu)}
\int_{\mathbb{R}^{2d}\times\mathbb{R}^{2d}}
|z-z'|\,\pi(dz,dz'),
\]
where $\Pi(\mu,\nu)$ consists of probability measures on
$\mathbb{R}^{2d}\times\mathbb{R}^{2d}$ with marginals $\mu$ and $\nu$.
For a measurable map
$\mathcal{T}:\mathbb{R}^{2d}\to\mathbb{R}^{2d}$, we write
$\mathcal{T}_{\#}\mu$ for the push-forward measure, defined by
$\mathcal{T}_{\#}\mu(B)=\mu(\mathcal{T}^{-1}(B))$
for every Borel set $B$. \\
Given the prescribed initial histories, we interpret each
kinetic equation in \eqref{pde1}--\eqref{pde2} in the weak sense;
see also \cite[Section~3]{CP}.
A curve
$\mu\in C([-\bar\tau,T);\mathcal{P}_1(\mathbb{R}^{2d}))$,
agreeing with the corresponding initial history, is a
measure-valued solution if
\begin{equation}\label{measure}
\begin{aligned}
&\int_{\mathbb{R}^{2d}}\varphi(x,v)\,\mu_t(dx,dv)
-\int_{\mathbb{R}^{2d}}\varphi(x,v)\,\mu_0(dx,dv)
\\
&\quad=
\int_0^t\int_{\mathbb{R}^{2d}}
\bigl(
v\cdot\nabla_x\varphi(x,v)
+F(s,x,v)\cdot\nabla_v\varphi(x,v)
\bigr)\,\mu_s(dx,dv)\,ds,
\end{aligned}
\end{equation}
for every $\varphi\in C_c^\infty(\mathbb{R}^{2d})$
and $t\in[0,T)$.
Here $F$ is the corresponding interaction field defined
in \eqref{flux1}, \eqref{flux2}, or \eqref{flux3},
evaluated using the prescribed histories and the evolving
populations. Continuity in $\mathcal{P}_1$ is understood
with respect to $d_1$. \\
We can now state the well-posedness and flocking results for the mean-field descriptions.

\begin{thm}[Global well-posedness]\label{limit_result}
Let $\psi,\phi,\rho:[0,\infty)\to(0,\infty)$ be bounded
and Lipschitz continuous. Assume that the delays $\tau_1,\tau_2$
and the nonnegative memory kernel $G$ satisfy the delay and
memory-kernel hypotheses given by \eqref{delay_ass} and \eqref{memory_cond}. Consider one of the following initial histories.
\begin{enumerate}
\item[\textup{(i)}]
For the hybrid system \eqref{leaders}--\eqref{pde1}, let
\[
(y_i^0,w_i^0)\in
C([-\bar\tau,0];\mathbb{R}^{2d}),\qquad i=1,\ldots,m,
\]
and
\[
s\mapsto\nu_s\in
C([-\bar\tau,0];\mathcal{P}_1(\mathbb{R}^{2d})).
\]
Assume that there exists $R_0>0$ such that, for every
$s\in[-\bar\tau,0]$,
\[
\operatorname{supp}\nu_s
\cup\{(y_i^0(s),w_i^0(s)):i=1,\ldots,m\}
\subseteq\overline{B}^{2d}(0,R_0).
\]

\item[\textup{(ii)}]
For the fully kinetic system \eqref{pde2}, let
\[
s\mapsto\bar\mu_s,\quad s\mapsto\bar\nu_s
\in C([-\bar\tau,0];\mathcal{P}_1(\mathbb{R}^{2d})).
\]
Assume that there exists $R_0>0$ such that
\[
\operatorname{supp}\bar\mu_s
\cup\operatorname{supp}\bar\nu_s
\subseteq\overline{B}^{2d}(0,R_0),
\qquad s\in[-\bar\tau,0].
\]
\end{enumerate}
Then, for any finite time $T>0$, the corresponding mean-field model admits a unique solution of equations \eqref{pde1} or \eqref{pde2} on the interval $[0,T)$ with the following regularity properties:
\begin{itemize}
    \item In Case (i), the leader trajectories $ y_i(t)$ belong to $C^1([0,T))$, and the follower distribution $f_t$ belongs to $C([0,T); \mathcal{P}_1(\mathbb{R}^{2d}))$, with uniformly compact support. Moreover, the measure-valued solution $f_t$ is the push-forward of the measure $\nu_0$ through the flow map generated by $(v,F^m[f]).$
    
    \item In Case (ii), the measure-valued solutions $g_t$ and $f_t$ belong to $C([0, T);  \mathcal{P}_1(\mathbb{R}^{2d}))$, with uniformly compact support, and they are push-forwards of the measure $\bar{\mu}_0$ and $\bar{\nu}_0$ through 
 the flow maps associated with $(w, F_L[g])$ and $(v, F_F[f,g])$, respectively.
\end{itemize}\end{thm}

\begin{thm}[Unconditional flocking]\label{limit_flocking}
Under the assumptions of Theorem~\ref{limit_result},
assume in addition condition \eqref{PE_memory} and the integral condition \eqref{infint}
holds. Then the corresponding global solution exhibits
unconditional flocking. More precisely, in each case there exists $\gamma>0$,
depending on the model parameters and the initial histories,
but independent of time, such that
\begin{enumerate}
\item[\textup{(i)}]
For the hybrid system \eqref{leaders}--\eqref{pde1},
\[
\sup_{t\geq0}d_X^f(t)<\infty,
\qquad
d_V^f(t)\leq D_0^f e^{-\gamma(t-2\bar\tau)}
\quad\text{for all }t\geq0.
\]
\item[\textup{(ii)}]
For the fully kinetic system \eqref{pde2},
\[
\sup_{t\geq0}d_X^{f,g}(t)<\infty,
\qquad
d_V^{f,g}(t)\leq
D_0^{f,g} e^{-\gamma(t-2\bar\tau)}
\quad\text{for all }t\geq0.
\]
\end{enumerate}
\end{thm}

The rest of this paper is organized as follows. In Section \ref{uncontditional}, we provide the flocking result for the particle model with time-delayed interactions and leadership under appropriate assumptions. Section \ref{wellp} is devoted to the mean-field formulation and the rigorous construction of measure-valued solutions to the limiting system, including the proof of Theorem \ref{limit_result}. In Section \ref{stability}, we analyze the stability and asymptotic behavior of the mean-field system. We first derive stability estimates in the Wasserstein distance, and then rigorously establish the flocking behavior by combining stability with particle approximation techniques, proving Theorem \ref{limit_flocking}.

\section{Unconditional flocking dynamics}\label{uncontditional}
	\subsection{Preliminaries}
Let $(y_i,w_i)_{i=1}^m$ and $(x_i,v_i)_{i=1}^N$
solve \eqref{leaders} and \eqref{followers}, with initial
conditions \eqref{incondcs_leaders} and
\eqref{incondcs_followers}, respectively.
Assume that the hypotheses of Theorem \ref{uf} hold. \begin{defn}\label{quant}
		Given a vector $v\in \mathbb{R}^d$, for all $n\in \mathbb{N}_0$ we define
	
		$$m_n^v:=\min_{s\in [n\bar\tau-\bar\tau,n\bar\tau]} \Big\{\min_{i=1,\dots,m}\,\langle w_{i}(s),v\rangle, \min_{i=1,\dots,N}\,\langle v_{i}(s),v\rangle \Big\}, $$
		$$M_n^v:=\max_{s\in [n\bar\tau-\bar\tau,n\bar\tau]}\Big\{\max_{i=1,\dots,m}\,\langle w_{i}(s),v\rangle, \max_{i=1,\dots,N}\,\langle v_{i}(s),v\rangle\Big\}.$$
	\end{defn} 
Notice that, from Definition \ref{Dn} and Cauchy-Schwarz inequality, for a unit vector $v \in \mathbb{R}^d$ we have that $M_n-m_n \le D_n.$
\begin{prop}[Preliminary estimates]\label{prop:preliminary}
The following estimates hold.
\begin{enumerate}
\item[(i)]
For every $v\in\mathbb{R}^d$ and $n\in\mathbb{N}_0$,
\begin{equation}\label{ineq1}
\begin{aligned}
m_n^v&\leq\langle w_i(t),v\rangle\leq M_n^v,
&& i=1,\ldots,m,\\
m_n^v&\leq\langle v_j(t),v\rangle\leq M_n^v,
&& j=1,\ldots,N,
\end{aligned}
\end{equation}
for all $t\geq n\bar\tau-\bar\tau$.

\item[(ii)]
For every $n\in\mathbb{N}_0$,
\begin{equation}\label{diam}
\begin{aligned}
|w_i(s)-w_j(t)|&\leq D_n,
&& i,j=1,\ldots,m,\\
|v_i(s)-v_j(t)|&\leq D_n,
&& i,j=1,\ldots,N,\\
|w_i(s)-v_j(t)|&\leq D_n,
&& i=1,\ldots,m,\quad j=1,\ldots,N,
\end{aligned}
\end{equation}
for all $s,t\geq n\bar\tau-\bar\tau$.
Consequently,
\begin{align}
d_V(t)&\leq D_n,
&& t\geq n\bar\tau-\bar\tau,
\label{diamonoff2}\\
D_{n+1}&\leq D_n,
&& n\in\mathbb{N}_0.
\label{decreasing}
\end{align}

\item[(iii)]
For all $t\geq-\bar\tau$,
\begin{equation}\label{bounds}
\begin{aligned}
|w_i(t)|&\leq C_0^V,
&& i=1,\ldots,m,\\
|v_j(t)|&\leq C_0^V,
&& j=1,\ldots,N,
\end{aligned}
\end{equation}
where $C_0^V$ is defined in \eqref{C0}.

\item[(iv)]
For every $t\geq0$,
\begin{align}
|y_i(t)-y_j(s)|
&\leq \bar\tau C_0^V+M_0^X+ d_X(t),
&& s\in[t-\tau_1(t),t], \ i,j=1, \dots,m,
\label{dist1}\\
|x_i(t)-x_j(s)|
&\leq \bar\tau C_0^V+M_0^X+ d_X(t),
&& s\in[t-\tau_2(t),t], \ i,j=1,\dots,N,
\label{dist2}\\
|x_i(t)-y_j(s)|
&\leq \bar\tau C_0^V+M_0^X+ d_X(t),
&& s\in[t-\tau_1(t),t], \ i=1,\dots,N, \ j=1,\dots,m,
\label{dist3}
\end{align}
where $M_0^X$ is defined in \eqref{MX}.
\end{enumerate}
\end{prop}

\begin{proof}
\emph{(i)}
Fix $n\in\mathbb{N}_0$ and $v\in\mathbb{R}^d$.
The inequalities in \eqref{ineq1} hold on
$[n\bar\tau-\bar\tau,n\bar\tau]$ by the definitions
of $m_n^v$ and $M_n^v$.
For $\epsilon>0$, define
\[
\mathcal{K}^{\epsilon}:=
\Bigl\{t>n\bar\tau:\; 
\max\Bigl\{
\max_{i=1,\ldots,m}\langle w_i(s),v\rangle,\,
\max_{j=1,\ldots,N}\langle v_j(s),v\rangle
\Bigr\} <M_n^v+\epsilon,\quad
\forall\,s\in[n\bar\tau,t)
\Bigr\},
\]
and let $S^\epsilon:=\sup\mathcal{K}^\epsilon$.
At time $n\bar\tau$, every projection is at most $M_n^v$.
Continuity of the finitely many velocity trajectories therefore
implies that $\mathcal{K}^\epsilon\ne\emptyset$ and
$S^\epsilon>n\bar\tau$. We claim that $S^\epsilon=+\infty$.
Suppose, by contradiction, that $S^\epsilon<+\infty$.
By definition,
\begin{equation}\label{max}
\max\Bigl\{
\max_{i=1,\ldots,m}\langle w_i(t),v\rangle,\,
\max_{j=1,\ldots,N}\langle v_j(t),v\rangle
\Bigr\}
<M_n^v+\epsilon
\end{equation}
for every $t\in[n\bar\tau,S^\epsilon)$.
Moreover, continuity and the definition of $S^\epsilon$ give
\begin{equation}\label{teps}
\max\Bigl\{
\max_{i=1,\ldots,m}\langle w_i(S^\epsilon),v\rangle,\,
\max_{j=1,\ldots,N}\langle v_j(S^\epsilon),v\rangle
\Bigr\}
=M_n^v+\epsilon.
\end{equation}
For $t\in[n\bar\tau,S^\epsilon)$, the delay bound implies
\[
[t-\tau_1(t),t]\cup[t-\tau_2(t),t]
\subseteq[n\bar\tau-\bar\tau,S^\epsilon).
\]
On $[n\bar\tau-\bar\tau,n\bar\tau]$ the projections are
bounded above by $M_n^v$ by definition; on
$[n\bar\tau,S^\epsilon)$ they satisfy \eqref{max}.
Thus all delayed velocities entering the equations obey
\begin{equation}\label{t-tau}
\begin{aligned}
\langle w_j(s),v\rangle&\leq M_n^v+\epsilon,
&& j=1,\ldots,m,\quad s\in[t-\tau_1(t),t],\\
\langle v_j(s),v\rangle&\leq M_n^v+\epsilon,
&& j=1,\ldots,N,\quad s\in[t-\tau_2(t),t].
\end{aligned}
\end{equation}
For every leader $i$, \eqref{leaders} and \eqref{t-tau} yield
\[
\begin{aligned}
\frac{d}{dt}\langle w_i(t),v\rangle
&=\frac{1}{m}\sum_{j=1}^m
\int_{t-\tau_1(t)}^t
G(t-s)c_{ij}(t,s)
\langle w_j(s)-w_i(t),v\rangle\,ds\\
&\leq
\bigl(M_n^v+\epsilon-\langle w_i(t),v\rangle\bigr)
\frac{1}{m}\sum_{j=1}^m
\int_{t-\tau_1(t)}^tG(t-s)c_{ij}(t,s)\,ds\\
&\leq
K\bigl(M_n^v+\epsilon-\langle w_i(t),v\rangle\bigr).
\end{aligned}
\]
The last step follows from \eqref{memory_cond} and \eqref{K}.
The factor multiplying the interaction rate is positive
by \eqref{max}. Similarly, the two interaction terms in \eqref{followers}
give, for each follower $i$,
\[
\begin{aligned}
\frac{d}{dt}\langle v_i(t),v\rangle
&\leq
\bigl(M_n^v+\epsilon-\langle v_i(t),v\rangle\bigr)
\Biggl[
\frac{1}{N}\sum_{j=1}^N
\int_{t-\tau_2(t)}^tG(t-s)a_{ij}(t,s)\,ds\\
&\qquad\qquad+
\frac{1}{m}\sum_{j=1}^m
\int_{t-\tau_1(t)}^tG(t-s)b_{ij}(t,s)\,ds
\Biggr]\\
&\leq
2K\bigl(M_n^v+\epsilon-\langle v_i(t),v\rangle\bigr).
\end{aligned}
\]
Each of the two interaction rates is bounded above by $K$,
which explains the factor $2K$.

Applying Gronwall's inequality from time $n\bar\tau$ gives
\[
\begin{aligned}
\langle w_i(t),v\rangle
&\leq
e^{-K(t-n\bar\tau)}\langle w_i(n\bar\tau),v\rangle +
K(M_n^v+\epsilon)
\int_{n\bar\tau}^t e^{-K(t-s)}\,ds\\
&=
M_n^v+\epsilon
-e^{-K(t-n\bar\tau)}
\bigl(M_n^v+\epsilon-\langle w_i(n\bar\tau),v\rangle\bigr)\\
&\leq M_n^v+\epsilon-\epsilon e^{-K(t-n\bar\tau)},
\end{aligned}
\]
where we used $\langle w_i(n\bar\tau),v\rangle\leq M_n^v$.
For the followers, the same computation yields
\[
\begin{aligned}
\langle v_i(t),v\rangle
&\leq
e^{-2K(t-n\bar\tau)}\langle v_i(n\bar\tau),v\rangle +
2K(M_n^v+\epsilon)
\int_{n\bar\tau}^t e^{-2K(t-s)}\,ds\\
&=
M_n^v+\epsilon
-e^{-2K(t-n\bar\tau)}
\bigl(M_n^v+\epsilon-\langle v_i(n\bar\tau),v\rangle\bigr)\\
&\leq M_n^v+\epsilon-\epsilon e^{-2K(t-n\bar\tau)}.
\end{aligned}
\]
Consequently, for every $t\in[n\bar\tau,S^\epsilon)$,
\begin{equation}\label{lim}
\max\Bigl\{
\max_{i=1,\ldots,m}\langle w_i(t),v\rangle,\,
\max_{j=1,\ldots,N}\langle v_j(t),v\rangle
\Bigr\} \leq
M_n^v+\epsilon-\epsilon e^{-2K(S^\epsilon-n\bar\tau)}.
\end{equation}
Letting $t\uparrow S^\epsilon$ in \eqref{lim} and using
\eqref{teps} gives the contradiction
\[
M_n^v+\epsilon
\leq M_n^v+\epsilon-\epsilon e^{-2K(S^\epsilon-n\bar\tau)}
<M_n^v+\epsilon.
\]
Therefore $S^\epsilon=+\infty$.
Since $\epsilon>0$ is arbitrary, the upper bounds in
\eqref{ineq1} follow for all $t\geq n\bar\tau$.

To obtain the lower bounds, apply the upper bounds to $-v \in \mathbb{R}^d$.
The definitions, with both populations included, imply
$M_n^{-v}=-m_n^v$.
Hence
\[
\begin{aligned}
-\langle w_i(t),v\rangle
&=\langle w_i(t),-v\rangle
\leq M_n^{-v}=-m_n^v,\\
-\langle v_j(t),v\rangle
&=\langle v_j(t),-v\rangle
\leq M_n^{-v}=-m_n^v.
\end{aligned}
\]
Multiplication by $-1$ gives the required lower bounds.
Together with the bounds on the defining history window,
this proves (i) for all $t\geq n\bar\tau-\bar\tau$.
In particular, $n=0$ recovers the bounds determined by
the initial history on $[-\bar\tau,0]$.

\medskip
\noindent\emph{(ii) Velocity diameters and monotonicity.}
Fix $n\in\mathbb{N}_0, $ and let us prove the first estimate in \eqref{diam}. The other two follow analogously. Let $s,t\geq n\bar\tau-\bar\tau$ and $i,j=1,\ldots,m$. If $w_i(s)=w_j(t)$, the inequality is immediate.
Otherwise, choose the unit vector
\[
v=\frac{w_i(s)-w_j(t)}{|w_i(s)-w_j(t)|}.
\]
Using (i) separately at times $s$ and $t$, we obtain
\[
\begin{aligned}
|w_i(s)-w_j(t)|
&=\langle w_i(s),v\rangle-\langle w_j(t),v\rangle\\
&\leq M_n^v-m_n^v
\leq D_n.
\end{aligned}
\]
Taking $s=t$ in \eqref{diam} bounds each term defining
the velocity diameter and proves \eqref{diamonoff2}.
Moreover,
\[
[n\bar\tau,(n+1)\bar\tau]
\subseteq[n\bar\tau-\bar\tau,+\infty).
\]
Thus all three estimates in \eqref{diam} hold for any two
times in the window defining $D_{n+1}$.
Taking the maxima over those times and over the corresponding
agent indices gives $D_{n+1}\leq D_n$, proving
\eqref{decreasing}.

\medskip
\noindent\emph{(iii) Uniform velocity bounds.}
For every unit vector $v$ and $s\in[-\bar\tau,0]$,
the Cauchy--Schwarz inequality and \eqref{C0} give
\[
\langle w_i(s),v\rangle\leq |w_i(s)|\leq C_0^V,
\qquad
\langle v_j(s),v\rangle\leq |v_j(s)|\leq C_0^V.
\]
Taking the maxima over both populations yields $M_0^v\leq C_0^V$. \\
Fix $t\geq-\bar\tau$.
If $w_i(t)\ne0$, take $v=w_i(t)/|w_i(t)|$ in (i),
with $n=0$, to obtain
\[
|w_i(t)|=\langle w_i(t),v\rangle
\leq M_0^v\leq C_0^V.
\]
If $v_j(t)\ne0$, instead take $v=v_j(t)/|v_j(t)|$ and obtain
\[
|v_j(t)|=\langle v_j(t),v\rangle
\leq M_0^v\leq C_0^V.
\]
This proves \eqref{bounds}.

\medskip
\noindent\emph{(iv) Spatial distances on the delay windows.}
First of all, let us prove \eqref{dist1}. Let $t\geq0$, $i,j=1,\ldots,m$, and
$s\in[t-\tau_1(t),t]$.
The triangle inequality gives
\begin{equation}\label{split}
\begin{aligned}
|y_i(t)-y_j(s)|
&\leq |y_i(t)-y_j(t)|+|y_j(t)-y_j(s)|\\
&\leq d_X(t)+|y_j(t)-y_j(s)|.
\end{aligned}
\end{equation}
We estimate the last term by distinguishing whether $s$
belongs to the evolution interval or to the initial history. If $s\geq0$, the position equation and (iii) imply
\[
|y_j(t)-y_j(s)|
\leq\int_s^t |w_j(r)|\,dr
\leq (t-s)C_0^V
\leq\bar\tau C_0^V.
\]
If $s<0$, the condition
$s\geq t-\tau_1(t)$ implies $t<\tau_1(t)\leq\bar\tau$.
Splitting at time zero separates the prescribed initial
history from the subsequent evolution:
\[
\begin{aligned}
|y_j(t)-y_j(s)|
&\leq |y_j(t)-y_j(0)|+|y_j(0)-y_j(s)|\\
&\leq\int_0^t |w_j(r)|\,dr+M_0^X\\
&\leq tC_0^V+M_0^X
\leq\bar\tau C_0^V+M_0^X.
\end{aligned}
\]
Here \eqref{MX} controls the initial-history difference.
Thus, in either case,
\[
|y_j(t)-y_j(s)|\leq M_0^X+\bar\tau C_0^V,
\]
and substitution into \eqref{split} yields
\[
|y_i(t)-y_j(s)|
\leq d_X(t)+M_0^X+\bar\tau C_0^V.
\]
For the followers, we can repeat the same argument to get, for $i,j=1,\ldots, N,$ 
\[
|x_i(t)-x_j(s)| \leq d_X(t)+M_0^X+\bar\tau C_0^V.
\]
Finally, let $i=1,\ldots,N$, $j=1,\ldots,m$, and
$s\in[t-\tau_1(t),t]$.
The definition of the mixed spatial diameter gives
$|x_i(t)-y_j(t)|\leq d_X(t)$.
Using the estimate already proved for the leader's
temporal displacement, we therefore obtain
\begin{equation}\label{split2}
\begin{aligned}
|x_i(t)-y_j(s)|
&\leq |x_i(t)-y_j(t)|+|y_j(t)-y_j(s)|\\
&\leq d_X(t)+M_0^X+\bar\tau C_0^V.
\end{aligned}
\end{equation}
This concludes the proof.
\end{proof}

\subsection{The flocking estimate }
To prove the flocking result, we need a contraction estimate in order to have the bound of the space diameter $d_X(t)$ for $t \geq 0.$ First of all, let us give the following definitions:
\begin{equation}\label{lambda}
\Lambda(t):=\kappa \left(\bar \tau C_0^V+M_0^X+ \max_{s \in [-\bar \tau,t]} d_X(s)\right), \qquad t \geq \bar \tau,
\end{equation}
	where $\kappa$ is defined in \eqref{infint}.

		\begin{oss}\label{lower}
		Let us note that from Proposition \ref{prop:preliminary} \emph{(iv)}, for all $t \geq 0,$
		$$\psi(\lvert y_i(t)-y_j(s)\rvert)\geq \Lambda(t),\quad i,j=1,\dots,m, \ s \in [t-\tau_1(t),t], $$
		$$\phi(\lvert x_i(t)-x_j(s)\rvert)\geq \Lambda(t),\quad i,j=1,\dots,N, \ s \in [t-\tau_2(t),t],$$
		and 
		$$\rho(\lvert x_i(t)-y_j(s)\rvert)\geq \Lambda(t),\quad i=1,\dots,N, \, j=1,\dots,m, \ s \in [t-\tau_1(t),t].$$
		In particular, these lower bounds hold in general for $s \in [t-\bar{\tau}, t],$ for all $t \geq 0.$
		\end{oss}

\begin{lem} \label{2.6}
For any unit vector $ v \in \mathbb{R}^d$ and $n \in \mathbb{N}_0$ we have that 
\begin{equation}
\begin{split}
& \langle w_i(t)-w_j(t),v \rangle \leq e^{-2K(t-t_0)} \langle w_i(t_0)-w_j(t_0),v \rangle + \left(1-e^{-2K(t-t_0)}\right)D_n, \\
& \hspace{12cm} i,j=1,\dots,m, \\
& \langle v_i(t)-v_j(t),v \rangle \leq e^{-2K(t-t_0)} \langle v_i(t_0)-v_j(t_0),v \rangle + \left(1-e^{-2K(t-t_0)}\right)D_n, \\
& \hspace{12cm} i,j=1,\dots,N, \\
& \langle v_i(t)-w_j(t),v \rangle \leq e^{-2K(t-t_0)} \langle v_i(t_0)-w_j(t_0),v \rangle + \left(1-e^{-2K(t-t_0)}\right)D_n, \\
& \hspace{10cm} i=1,\dots,N, \, j=1,\dots,m,
\end{split}
\label{eq18}
\end{equation}
for all $t \geq t_0 \geq n\bar\tau.$ Moreover, for all $n \in \mathbb{N}_0$ we get
\begin{equation}
D_{n+1} \leq e^{-2K \bar\tau} d_V(n\bar\tau) + \left(1-e^{-2K\bar\tau}\right)D_n. 
\label{eq20}
\end{equation}
\end{lem}

\begin{proof} We divide the proof into two steps. In the first step, we obtain contraction estimates for the differences among agents within the same group (both followers or leaders), and in the second step, we treat the mixed case involving a follower and a leader. \\

\emph{\textbf{Step 1.}} We first derive the contraction estimate for the follower agents. Fix a unit vector $v \in \mathbb{R}^d$ and a given $n \in \mathbb{N}_0$. Moreover, fix an index $i \in \{1,\dots,N\}$ and consider $t \geq t_0 \geq n\bar\tau$. \\ 
By the definition of system \eqref{followers} and applying Proposition \ref{prop:preliminary} \emph{(i)}, we have
\begin{equation*}
\begin{split}
\frac{d}{d t} \langle v_i(t),v \rangle & = \frac{1}{N} \sum_{j=1}^N \int_{t-\tau_2(t)}^t G(t-s)a_{ij}(t,s) \langle v_j(s)-v_i(t),v \rangle  \ ds \\
 & \hspace{3cm}+ \frac{1}{m} \sum_{j=1}^{m} \int_{t-\tau_1(t)}^t G(t-s) b_{ij}(t,s) \langle w_j(s)-v_i(t),v \rangle \ ds \\
%& \leq \frac{1}{N}\sum_{j=1}^N \int_{t-\tau_2(t)}^t G(t-s)a_{ij}(t,s) \bigl(M_n^v - \langle v_i(t),v \rangle\bigr) \ ds \\
%& \hspace{3cm}+ \frac{1}{m} \sum_{j=1}^{m} \int_{t-\tau_1(t)}^t G(t-s)b_{ij}(t,s) \bigl(M_{n}^v- \langle v_i(t),v \rangle \bigr) \ ds\\
%& \leq \frac{K}{N}\sum_{j=1}^N \int_{t-\tau_2(t)}^t G(t-s)\bigl(M_n^v - \langle v_i(t),v \rangle\bigr) \ ds \\
%& \hspace{3cm}+ \frac{K}{m} \sum_{j=1}^{m} \int_{t-\tau_1(t)}^t G(t-s) \bigl(M_{n}^v- \langle v_i(t),v \rangle \bigr) \ ds\\
& \leq 2K\bigl(M_n^v - \langle v_i(t),v \rangle\bigr), \\
\end{split}
\end{equation*}
where we used that $a_{ij}(t,s)$ and $b_{ij}(t,s)$ are bounded by $K$, the assumption \eqref{memory_cond}, and that $t-\tau_1(t),\,t-\tau_2(t) \geq n\bar\tau-\bar\tau$. By applying the Gr\"onwall's lemma, we find 
\begin{equation}\label{eq21}
\langle v_i(t),v \rangle \leq e^{-2K(t-t_0)}  \langle v_i(t_0),v \rangle + \left(1-e^{-2K(t-t_0)}\right)M_n^v .
\end{equation}
Similarly, for any $j \in \{1,\dots,N\}$ and $t \geq t_0 \geq n\bar\tau$, we derive the lower bound
\begin{equation}\label{eq22}
\langle v_j(t),v \rangle \geq e^{-2K(t-t_0)}  \langle v_j(t_0),v \rangle + \left(1-e^{-2K(t-t_0)}\right)m_n^v .
\end{equation}
Subtracting \eqref{eq22} from \eqref{eq21} yields
\begin{equation*}
\langle v_i(t)-v_j(t),v \rangle  \leq e^{-2K(t-t_0)}\langle v_i(t_0)-v_j(t_0),v \rangle + \left(1-e^{-2K(t-t_0)}\right)D_n.
\end{equation*}
Thus, we have obtained the contraction estimate for the followers as stated in \eqref{eq18}. By an analogous argument, one obtains a similar contraction estimate for the leader dynamics. \\

\emph{\textbf{Step 2.}} We now consider the mixed case involving a follower and a leader. From \eqref{leaders}, using again Proposition \ref{prop:preliminary} \emph{(i)}, we obtain
\begin{equation*}
\begin{split}
\frac{d}{d t} \langle w_j(t),v \rangle & = \frac{1}{m} \sum_{l=1}^m \int_{t-\tau_1(t)}^t G(t-s)c_{jl}(t,s) \langle w_l(s)-w_j(t),v \rangle \ ds \\
%& \geq \frac{1}{m}\sum_{l=1}^m  \int_{t-\tau_1(t)}^t G(t-s)c_{jl}(t,s) (m_n^v - \langle w_j(t),v \rangle) \ ds \\
& \geq K(m_n^v - \langle w_j(t),v \rangle) \geq 2K(m_n^v - \langle w_j(t),v \rangle),
\end{split}
\end{equation*}
where we used $m_n^v - \langle w_j(t),v \rangle\le 0$ for all $j=1,\dots,m$ and for all $t\ge n\bar\tau.$ Applying Gr\"onwall's inequality then gives
\begin{equation}\label{eq22A}
\langle w_j(t),v \rangle \geq e^{-2K(t-t_0)}  \langle w_j(t_0),v \rangle + \left(1-e^{-2K(t-t_0)}\right)m_n^v.
\end{equation}
Subtracting \eqref{eq22A} from the bound for $\langle v_i(t), v \rangle$ in \eqref{eq21} yields the estimate.
%\begin{equation*}
%\begin{split}
%\langle v_i(t)-w_j(t),v \rangle & \leq e^{-2K(t-t_0)}  \langle v_i(t_0)-w_j(t_0),v \rangle + \left(1-e^{-2K(t-t_0)}\right)(M_n^v-m_n^v) \\
%& \leq e^{-2K(t-t_0)}\langle v_i(t_0)-w_j(t_0),v \rangle + \left(1-e^{-2K(t-t_0)}\right)D_n.
%\end{split}
%\end{equation*}
This completes the proof of \eqref{eq18}. \\
To prove \eqref{eq20}, let us assume, without loss of generality, that for $t_1, t_2 \in [n\bar\tau, n\bar\tau+\bar\tau],$ $ D_{n+1}= \lvert w_i(t_1)-w_j(t_2)\rvert,$
for $i,j=1,\dots,m.$ Notice that, if $D_{n+1}=0$ the inequality is trivial. Therefore, assume $D_{n+1}>0,$ and consider the unit vector given by 
$$ v:=\frac{w_i(t_1)-w_j(t_2)}{\lvert w_i(t_1)-w_j(t_2)\rvert}.$$
From this definition, applying the computation in the previous steps (in particular, an adaptation to this case of  \eqref{eq21} and \eqref{eq22A}), we find
\begin{equation*}
\begin{split}
D_{n+1}=\langle w_i(t_1)-w_j(t_2),v \rangle & \leq e^{-2K\bar\tau}\lvert w_i(n\bar\tau)-w_j(n\bar\tau) \rvert + \left(1-e^{-2K\bar\tau}\right)D_n \\
& \leq e^{-2K\bar\tau}d_V(n\bar\tau) + \left(1-e^{-2K\bar\tau}\right)D_n.
\end{split}
\end{equation*}
The same argument applies when \(D_{n+1}\) is attained by a pair of followers or by a leader–follower pair, since the upper and lower projection bounds with the common coefficient \(e^{-2K\bar\tau}\) hold for both populations. This completes the proof.
\end{proof}

\begin{oss}
It is worth noting that for pairs of leaders, one may also derive a sharper contraction estimate that depends only on the leader group:
\begin{equation*}
\begin{split}
\langle w_i(t)-w_j(t),v \rangle & \leq e^{-K(t-t_0)} \langle w_i(t_0)-w_j(t_0),v \rangle \\ 
& \hspace{1.5cm} + \left(1-e^{-K(t-t_0)}\right)\max_{h,k=1,\dots, m}\max_{r,s\in [n\bar\tau-\bar\tau, n\bar\tau]} \vert w_h(r)-w_k(s)\vert,
\label{5D}
\end{split}
\end{equation*}
 for all $i,j=1,\dots,m$. However, to establish velocity alignment of the entire leader-follower population, it is essential to work with unified estimates (as in Lemma \ref{2.6}) that simultaneously control all interactions in the mixed leader-follower system.
\end{oss}
Let us consider the following definition: 
\begin{equation}\label{gn}
\Gamma_n:= \Lambda(n\bar\tau)e^{-4K\bar\tau}\eta, 
\end{equation}
where $K$ and $\Lambda(\cdot)$ are defined in \eqref{K} and  \eqref{lambda} respectively. Notice that $\Gamma_n \in (0,1),$ for all $n \in \mathbb{N}_0.$

\begin{lem} \label{6}
It holds that
\begin{equation} \label{estimate_diam1}
D_{n+1} \leq \left(1-\Gamma_n\right) D_{n-2},
\end{equation}
for all $n \geq 2.$
\end{lem}
\begin{proof} 
First of all, we claim that 
\begin{equation}\label{claim}
d_V(n\bar\tau) \leq C(n) D_{n-2},
\end{equation}
for a suitable coefficient $C(n) \in (0,1),$ and for all $n \ge 2.$
We prove the claim by considering three distinct cases, each corresponding to a different configuration in which the diameter $d_V(n\bar\tau)$ is achieved, according to the definition \eqref{diamV}.

\medskip

\textbf{\emph{Case 1.}} Assume that 
\[
d_V(n\bar\tau)=|v_i(n\bar\tau)-v_j(n\bar\tau)| 
\]
for some $i,j=1,\dots,N.$ Since the case $|v_i(n\bar\tau)-v_j(n\bar\tau)|=0$ is trivial, we suppose $|v_i(n\bar\tau)-v_j(n\bar\tau)|>0.$ In this case, we first normalize the difference by setting
\[
v:=\frac{v_i(n\bar\tau)-v_j(n\bar\tau)}{|v_i(n\bar\tau)-v_j(n\bar\tau)|}.
\]
So that, we have
$$ d_V(n\bar\tau)= \langle v_i(n\bar\tau)-v_j(n\bar\tau),v \rangle.$$
We analyze the evolution of the projection differences along $v \in \mathbb{R}^d$ during the time interval $t \in [(n-1)\bar\tau,n\bar\tau]$. Using the system dynamics and Definition \ref{quant}, we
%\begin{align*}
% \frac{d}{d t} \langle v_i(t)-v_j(t), v \rangle & = \frac{1}{N} \sum_{l=1}^N \int_{t-\tau_2(t)}^t G(t-s) a_{il}(t,s) \langle v_l(s)-v_i(t), v \rangle \ ds \cr
%& \hspace{3cm}+ \frac{1}{m}\sum_{l=1}^{m} \int_{t-\tau_1(t)}^t G(t-s) b_{il}(t,s) \langle w_l(s)-v_i(t),v \rangle \ ds \\
%& - \frac{1}{N}\sum_{l=1}^N \int_{t-\tau_2(t)}^t G(t-s) a_{jl}(t,s) \langle v_l(s)-v_j(t), v \rangle \ ds \cr
%& \hspace{3cm} - \frac{1}{m}\sum_{l=1}^{m} \int_{t-\tau_1(t)}^t G(t-s) b_{jl}(t,s) \langle w_l(s)-v_j(t), v \rangle \ ds.
%\end{align*}
regroup the terms by introducing the shift $M_{n-1}^v$ and $m_{n-1}^v,$ which represent an upper and a lower bound on the projections during $[(n-2)\bar\tau,(n-1)\bar\tau].$ We rewrite 
\begin{equation*}
\frac{d}{d t} \langle v_i(t)-v_j(t), v \rangle  = S_1 + S_2,
\end{equation*}
%\begin{equation*}
%\begin{split}
%\frac{d}{d t} & \langle v_i(t)-v_j(t), v \rangle  = \frac{1}{N}\sum_{l=1}^N \int_{t-\tau_2(t)}^t G(t-s) a_{il}(t,s) (\langle v_l(s), v \rangle - M_{n-1}^v) \ ds \cr 
%& \hspace{3cm} + \frac{1}{N}\sum_{l=1}^N \int_{t-\tau_2(t)}^t G(t-s)a_{il}(t,s) (M_{n-1}^v-\langle v_i(t), v \rangle) \ ds \\
%&+ \frac{1}{m} \sum_{l=1}^{m} \int_{t-\tau_1(t)}^t G(t-s)b_{il}(t,s) (\langle w_l(s), v \rangle - M_{n-1}^v) \ ds \cr
%& \hspace{3cm} + \frac{1}{m} \sum_{l=1}^{m} \int_{t-\tau_1(t)}^t G(t-s) b_{il}(t,s) (M_{n-1}^v- \langle v_i(t), v \rangle) \ ds \\
%\end{split}
%\end{equation*}
%
%\begin{equation*}
%\begin{split}
%& + \frac{1}{N}\sum_{l=1}^N \int_{t-\tau_2(t)}^t G(t-s) a_{jl}(t,s) (m_{n-1}^v-\langle v_l(s), v \rangle) \ ds \\
%& \hspace{3 cm}+ \frac{1}{N}\sum_{l=1}^N \int_{t-\tau_2(t)}^t G(t-s) a_{jl}(t,s) (\langle v_j(t), v \rangle-m_{n-1}^v) \ ds \\
%& + \frac{1}{m}\sum_{l=1}^m \int_{t-\tau_1(t)}^t G(t-s)b_{jl}(t,s) (m_{n-1}^v-\langle w_l(s), v \rangle) \ ds \\
%& \hspace{3cm} + \frac{1}{m}\sum_{l=1}^m \int_{t-\tau_1(t)}^t G(t-s) b_{jl}(t,s) (\langle v_j(t), v \rangle-m_{n-1}^v) \ ds .
%\end{split}
%\end{equation*}
where the two auxiliary sums $S_1$ and $S_2$ correspond to the contributions involving the upper bound $M_{n-1}^v$ and the lower bound $m_{n-1}^v$ respectively. In particular, we have that
\begin{equation*}
\begin{split}
S_1 & := \frac{1}{N}\sum_{l=1}^N \int_{t-\tau_2(t)}^t G(t-s) a_{il}(t,s) (\langle v_l(s), v \rangle - M_{n-1}^v) \ ds \cr 
& \hspace{3cm} + \frac{1}{N}\sum_{l=1}^N \int_{t-\tau_2(t)}^t G(t-s)a_{il}(t,s) (M_{n-1}^v-\langle v_i(t), v \rangle) \ ds \\
&+ \frac{1}{m} \sum_{l=1}^{m} \int_{t-\tau_1(t)}^t G(t-s)b_{il}(t,s) (\langle w_l(s), v \rangle - M_{n-1}^v) \ ds \cr
& \hspace{3cm} + \frac{1}{m} \sum_{l=1}^{m} \int_{t-\tau_1(t)}^t G(t-s) b_{il}(t,s) (M_{n-1}^v- \langle v_i(t), v \rangle) \ ds 
\end{split}
\end{equation*}

\begin{equation*}
\begin{split}
& \leq \frac{\Lambda(n\bar\tau)}{N} \sum_{l=1}^N \int_{t-\tau_2(t)}^t G(t-s)(\langle v_l(s), v \rangle - M_{n-1}^v) \ ds \\
& \hspace{3cm}+ \frac{\Lambda(n\bar\tau)}{m} \sum_{l=1}^{m} \int_{t-\tau_1(t)}^t G(t-s)(\langle w_l(s), v \rangle - M_{n-1}^v) \ ds \\
& + 2K (M_{n-1}^v-\langle v_i(t), v \rangle),
\end{split}
%\label{s1}
\end{equation*} 
where we used  Remark \ref{lower} and \eqref{K}.  Similarly, we find

\begin{equation*}
\begin{split}
S_2 & := \frac{1}{N}\sum_{l=1}^N \int_{t-\tau_2(t)}^t G(t-s) a_{jl}(t,s) (m_{n-1}^v-\langle v_l(s), v \rangle) \ ds \\
& \hspace{3 cm}+ \frac{1}{N}\sum_{l=1}^N \int_{t-\tau_2(t)}^t G(t-s) a_{jl}(t,s) (\langle v_j(t), v \rangle-m_{n-1}^v) \ ds \\
& + \frac{1}{m}\sum_{l=1}^m \int_{t-\tau_1(t)}^t G(t-s)b_{jl}(t,s) (m_{n-1}^v-\langle w_l(s), v \rangle) \ ds 
\end{split}
\end{equation*}

\begin{equation*}
\begin{split}
& \hspace{3cm} + \frac{1}{m}\sum_{l=1}^m \int_{t-\tau_1(t)}^t G(t-s) b_{jl}(t,s) (\langle v_j(t), v \rangle-m_{n-1}^v) \ ds \\
& \leq \frac{\Lambda(n\bar\tau)}{N} \sum_{l=1}^N \int_{t-\tau_2(t)}^t G(t-s)(m_{n-1}^v-\langle v_l(s), v \rangle) \ ds \\
& \hspace{3 cm}+ \frac{\Lambda(n\bar\tau)}{m}\sum_{l=1}^m \int_{t-\tau_1(t)}^t G(t-s) (m_{n-1}^v-\langle w_l(s), v \rangle) \ ds \\
& + 2K(\langle v_j(t), v \rangle - m_{n-1}^v).
\end{split}
%\label{s2}
\end{equation*}
Combining the estimates from $S_1$ and $S_2,$ we deduce that
\begin{equation*}
  \frac{d}{d t} \langle v_i(t)-v_j(t), v \rangle \leq \left[ 2K-\Lambda(n\bar\tau)(g_1(t)+g_2(t))\right] (M_{n-1}^v-m_{n-1}^v) - 2K \langle v_i(t)-v_j(t), v \rangle.
\end{equation*}
Applying the Gr\"onwall's lemma on the time interval $[(n-1)\bar\tau,t]$ with $t \in [(n-1)\bar\tau,n\bar\tau],$ we find that
\begin{equation*}
\begin{split}
\langle v_i(t)-v_j(t), v \rangle & \leq e^{-2K(t-n\bar\tau+\bar\tau)} \langle v_i(n\bar\tau-\bar\tau)-v_j(n\bar\tau-\bar\tau), v \rangle \\
& \quad + (M_{n-1}^v-m_{n-1}^v) \int_{n\bar\tau-\bar\tau}^t e^{-2K(t-s)} \left( 2K- \Lambda(n\bar\tau)(g_1(s)+g_2(s)) \right) ds.
\end{split}
\end{equation*}
Since this is valid for all $\ t \in [(n-1)\bar\tau,n\bar\tau], $ we take $t=n\bar\tau.$ Moreover, using condition \eqref{PE_memory} and the fact that $g_2\geq 0,$ we obtain
$$ \int_{n\bar\tau-\bar\tau}^{n\bar\tau} e^{-2K(n\bar\tau-s)} (g_1(s)+g_2(s)) ds \geq e^{-2K\bar\tau}\eta.$$
Therefore, we find that
\begin{equation*} 
\begin{split} 
&\langle v_i(n\bar\tau)-v_j(n\bar\tau), v \rangle \cr
&\quad \leq e^{-2K\bar\tau} |v_i(n\bar\tau-\bar\tau)-v_j(n\bar\tau-\bar\tau)| |v| + (M_{n-1}^v-m_{n-1}^v) (1-e^{-2K\bar\tau}-\Lambda(n\bar\tau)e^{-2K\bar\tau}\eta)\\
&\quad \leq D_{n-1} \left [1-\Lambda(n\bar\tau)e^{-2K\bar\tau}\eta\right ] \leq D_{n-2} \left [1-\Lambda(n\bar\tau)e^{-2K\bar\tau}\eta \right ],
\end{split}
\end{equation*}
where we used the nonincreasing property of $D_n$. Consequently, we deduce that
$$d_V(n\bar\tau)\leq D_{n-2} \left [1-\Lambda(n\bar\tau)e^{-2K\bar\tau}\eta \right ].$$
Thus, we obtain the desired estimate for the case when the maximum diameter is determined by follower agents.

\medskip

\textbf{\emph{Case 2.}} Now, assume 
\[
d_V(n\bar\tau)=|w_i(n\bar\tau)-w_j(n\bar\tau)|,
\] 
for some $i,j=1,\dots,m$. We begin by normalizing the difference. Define
\[
v:=\frac{w_i(n\bar\tau)-w_j(n\bar\tau)}{|w_i(n\bar\tau)-w_j(n\bar\tau)|}.
\]
Applying the same method as in Case 1, for $t \in [(n-1)\bar\tau,n\bar\tau]$, 
%using the definition of the system, we write the time derivative of the projection difference between the leaders' velocities $w_i$ and $w_j$ as
%
%\begin{equation*}
%\begin{split}
%  \frac{d}{d t} \langle w_i(t)-w_j(t), v \rangle  
%&= \frac{1}{m} \sum_{l=1}^m \int_{t-\tau_1(t)}^t G(t-s) c_{il}(t,s) \langle w_l(s)-w_i(t), v \rangle \ ds \\ 
%& \hspace{3cm}- \frac{1}{m}\sum_{l=1}^{m} \int_{t-\tau_1(t)}^t G(t-s) c_{jl}(t,s) \langle w_l(s)-w_j(t),v \rangle \ ds \\
%& = \frac{1}{m}\sum_{l=1}^m \int_{t-\tau_1(t)}^t G(t-s) c_{il}(t,s) (\langle w_l(s), v \rangle - M_{n-1}^v) \ ds \\
%& \hspace{3cm} + \frac{1}{m}\sum_{l=1}^m \int_{t-\tau_1(t)}^t G(t-s)c_{il}(t,s) (M_{n-1}^v-\langle w_i(t), v \rangle) \\
%&  + \frac{1}{m}\sum_{l=1}^m \int_{t-\tau_1(t)}^t G(t-s) c_{jl}(t,s) (m_{n-1}^v-\langle w_l(s), v \rangle) \ ds \\
%& \hspace{3cm} + \frac{1}{m}\sum_{l=1}^m \int_{t-\tau_1(t)}^t G(t-s) c_{jl}(t,s) (\langle w_j(t), v \rangle-m_{n-1}^v)
%\end{split}
%\end{equation*}
we find
\begin{equation*}
  \frac{d}{d t} \langle w_i(t)-w_j(t), v \rangle \leq (M_{n-1}^v-m_{n-1}^v)\left[ K- \Lambda(n\bar\tau) g_1(t)\right] - K \langle w_i(t)-w_j(t), v \rangle.
\end{equation*}
Applying again the Gr\"onwall's lemma over the interval $[(n-1)\bar\tau,t],$ and taking $t = n\bar\tau,$ we find that
\[
\begin{split} 
& \langle w_i(n\bar\tau)-w_j(n\bar\tau), v \rangle  \cr
&\quad \leq e^{-K\bar\tau} \langle w_i(n\bar\tau-\bar\tau)-w_j(n\bar\tau-\bar\tau), v \rangle + (M_{n-1}^v-m_{n-1}^v) (1-e^{-K\bar\tau}-\Lambda(n\bar\tau)e^{-K\bar\tau}\eta) \\
&\quad \leq  D_{n-2} \left [1-\Lambda(n\bar\tau)e^{-K\bar\tau}\eta \right ].
\end{split}
\]%\end{equation}
Thus, 
\[%\begin{equation}
d_V(n\bar\tau) \leq D_{n-2}\left [1-\Lambda(n\bar\tau)e^{-K\bar\tau}\eta \right ].
\]%\end{equation} 
This completes the derivation for Case 2.

\medskip

\textbf{\em Case 3.} Now, assume that there exist indices $i\in\{1,\dots,N\}$ and $j\in\{1,\dots,m\}$ such that
 \[
d_V(n\bar\tau)= \vert v_i(n\bar\tau)-w_j(n\bar\tau)\vert. 
\] 
As before, we define
\[
v:=\frac{v_i(n\bar\tau)-w_j(n\bar\tau)}{|v_i(n\bar\tau)-w_j(n\bar\tau)|}.
\]
For $t\in[(n-1)\bar\tau,n\bar\tau]$, by using almost the same arguments used in the previous cases, we deduce 
\begin{equation*}
%\begin{split}
 \frac{d}{d t} \langle v_i(t)-w_j(t), v \rangle 
%&\leq 2K (M_{n-1}^v-m_{n-1}^v) - 2K \langle v_i(t)-w_j(t), v \rangle \\
%&  + \frac{\Lambda(n\bar\tau)}{N}  \sum_{l=1}^{N}\int_{t-\tau_2(t)}^t G(t-s)(\langle v_l(s), v \rangle - M_{n-1}^v) \ ds \\
%& \hspace{3cm} + \frac{\Lambda(n\bar\tau)}{m} \sum_{l=1}^{m}\int_{t-\tau_1(t)}^t G(t-s)(\langle w_l(s), v \rangle - M_{n-1}^v) \ ds \\
%&+ \frac{\Lambda(n\bar\tau)}{m} \sum_{l=1}^{m} \int_{t-\tau_1(t)}^t G(t-s) (m_{n-1}^v-\langle w_l(s), v \rangle) \ ds\\
\leq (M_{n-1}^v-m_{n-1}^v) \left(2K-\Lambda(n\bar\tau)g_1(t)\right)- 2K \langle v_i(t)-w_j(t), v \rangle,
%\end{split}
\end{equation*}
where we used that for all $l=1, \dots, N, $
\[
\langle v_l(s), v \rangle - M_{n-1}^v \le 0,
\]
and, we obtain 
\[%\begin{equation} \label{2.37}
d_V(n\bar\tau) \leq D_{n-2} \left [1-\Lambda(n\bar\tau)e^{-2K\bar\tau}\eta\right ].
\]%\end{equation} 
To complete the proof of the claim, we define
\begin{equation} \label{C}
 C(n) := 1-\Lambda(n\bar\tau)e^{-2K\bar\tau}\eta.
\end{equation}
Now, to prove \eqref{estimate_diam1}, we use inequality \eqref{eq20}. Indeed, we have 
\begin{equation*}
\begin{split}
& D_{n+1} \leq e^{-2K \bar\tau} d_V(n\bar\tau) + (1-e^{-2K\bar\tau})D_n \\
& \leq e^{-2K\bar\tau}C(n)D_{n-2}+(1-e^{-2K\bar\tau})D_n \\
& \leq e^{-2K\bar\tau}C(n)D_{n-2}+(1-e^{-2K\bar\tau})D_{n-2} \\
& = D_{n-2}\left[ 1-(1-C(n))e^{-2K\bar\tau}\right].
\end{split}
\end{equation*}
Notice that $\Gamma_n=(1-C(n))e^{-2K\bar\tau}=\Lambda(n\bar\tau)e^{-4K\bar\tau}\eta,$ we have the result.
\end{proof}
We are now able to prove Theorem \ref{uf}.

	\begin{proof}[Proof of Theorem \ref{uf}]
		Let $\{(y_{j},w_{j})\}_{j=1,\dots,m}$ and $\{(x_{i},v_{i})\}_{i=1,\dots,N}$ be solutions to \eqref{leaders} and \eqref{followers} respectively, under the initial conditions \eqref{incondcs_leaders}-\eqref{incondcs_followers}. 
		Let us introduce the function $\mathcal{D}:[-\bar\tau,+\infty)\rightarrow [0,+\infty),$ 
$$ \mathcal{D}(t):=\begin{cases}
			D_{0}, \hspace{6.5cm}t\in [-\bar\tau,2\bar\tau],\\
			\mathcal{D}(n\bar\tau)\exp\left(-\frac{\eta e^{-4K\bar\tau}}{3\bar\tau}\int_{n\bar\tau}^t\Lambda(s+2\bar\tau) ds \right), \quad t\in (n\bar\tau,(n+1)\bar\tau],\,n\geq 2.
		\end{cases}$$
		By definition, $\mathcal{D}$ is continuous, nonnegative, and nonincreasing. Moreover, we claim that 
		\begin{equation}\label{boundIn}
			D_{n}\leq \mathcal{D}(t),\quad \forall t\in [-\bar\tau,n\bar\tau],\,\forall n\in\mathbb{N}_0.
		\end{equation} 
We prove this by induction. For $n=0,1,2$, from \eqref{decreasing} and the definition of $\mathcal{D},$ we can immediately say that $$ D_n \leq D_0 = \mathcal{D}(t),\quad \forall  t\in [-\bar\tau,2\bar\tau].$$
		Now, assume that \eqref{boundIn} holds until $n,$ for some $n\geq 2$. We have to show that \eqref{boundIn} is true also for $n+1$. From the induction hypothesis and by using again \eqref{decreasing}, we have that 
		$$D_{n+1}\leq D_{n}\leq \mathcal{D}(t),$$
		for all $t\in [-\bar\tau,n\bar\tau]$. It remains to prove that $D_{n+1}\leq \mathcal{D}(t)$, for all $t\in (n\bar\tau, (n+1)\bar\tau]$. First of all, notice that iterating the definition of $\mathcal D,$ we can write 
\begin{equation}\label{iterazione}
\mathcal D(t) = \mathcal D((n-2)\bar\tau)\exp\left[ -\frac{\eta e^{-4K\bar\tau}}{3\bar\tau}\int_{\max\{2\bar\tau,(n-2)\bar\tau\}}^t \Lambda(s+2\bar\tau) \,ds \right] .
\end{equation}
		Since $\Lambda$ is non-increasing and $s+2\bar\tau \geq n\bar\tau,$ we have 
		$$ \int_{\max\{2\bar\tau,(n-2)\bar\tau\}}^t \Lambda(s+2\bar\tau) \,ds \leq 3\bar\tau \Lambda(n\bar\tau).$$
Therefore, from Lemma \ref{6} and \eqref{decreasing}, it comes that
		
		$$\mathcal{D}(t)\geq \mathcal{D}((n-2)\bar\tau)e^{-\Gamma_n} \geq \mathcal{D}((n-2)\bar\tau)\left(1-\Gamma_n\right) \geq D_{n-2}\left(1-\Gamma_n\right) \geq D_{n+1},$$
	for all $t\in (n\bar\tau,(n+1)\bar\tau]$, wherein the above inequalities we have used the fact that $\mathcal{D}$ is nonincreasing, the inequality $e^{-x}\geq 1-x$ and the induction hypothesis. Hence, \eqref{boundIn} is proven. \\
		Next, we define the function
$\mathcal W:[2\bar\tau,+\infty)\to[0,+\infty)$ by
\[
\mathcal W(t):=
\bar\tau\mathcal D(t)
+\frac{\eta e^{-4K\bar\tau}}{3}\int_0^{
\bar\tau C_0^V+M_0^X
+\max_{s\in[-\bar\tau,t+2\bar\tau]}d_X(s)}
\kappa(r)\,dr.
\]
The function $\mathcal D$ and the running maximum of $d_X$
are locally Lipschitz for nonnegative times. Moreover, since
$\kappa$ is bounded, its integral is Lipschitz. Therefore,
$\mathcal W$ is locally Lipschitz and, in particular, locally
absolutely continuous. For a.e.\ $t>0$, we have
(see \cite{Cont} for further details)
\begin{equation}\label{derdiamX}
0\le
\frac{d}{dt}\max_{s\in[-\bar\tau,t]}d_X(s)
\le
\left|\frac{d}{dt}d_X(t)\right|
\le d_V(t).
\end{equation}
Furthermore, by \eqref{boundIn} and the definition of the
local velocity diameters,
\[
d_V(t)\le\mathcal D(t),
\qquad t\ge0.
\]
Hence, from the definition of $\mathcal D$, for a.e.\
$t>2\bar\tau$,
\begin{align*}
\frac{d}{dt}\mathcal W(t)
&=
-\frac{\eta e^{-4K\bar\tau}}{3}
\Lambda(t+2\bar\tau)\mathcal D(t) +\frac{\eta e^{-4K\bar\tau}}{3}\Lambda(t+2\bar\tau)
\frac{d}{dt}
\max_{s\in[-\bar\tau,t+2\bar\tau]}d_X(s)
\\
&\le
\frac{\eta e^{-4K\bar\tau}}{3}
\Lambda(t+2\bar\tau)
\bigl[d_V(t+2\bar\tau)-\mathcal D(t)\bigr]
\\
&\le
\frac{\eta e^{-4K\bar\tau}}{3}
\Lambda(t+2\bar\tau)
\bigl[\mathcal D(t+2\bar\tau)-\mathcal D(t)\bigr] \le0,
\end{align*}
where the last inequality follows from the nonincreasing
property of $\mathcal D$. Thus,
\begin{equation}\label{negder}
\frac{d}{dt}\mathcal W(t)\le0,
\qquad\text{for a.e.\ }t>2\bar\tau.
\end{equation}
By local absolute continuity,
\begin{equation}\label{2tau1}
\mathcal W(t)\le\mathcal W(2\bar\tau),
\qquad t\ge2\bar\tau.
\end{equation}
Here,
\[
\mathcal W(2\bar\tau)
=
\bar\tau D_0
+\frac{\eta e^{-4K\bar\tau}}{3}\int_0^{
\bar\tau C_0^V+M_0^X
+\max_{s\in[-\bar\tau,4\bar\tau]}d_X(s)}
\kappa(r)\,dr
<+\infty.
\]
Therefore, since $\mathcal{D}$ is non-negative,
\[
\frac{\eta e^{-4K\bar\tau}}{3}\int_0^{
\bar\tau C_0^V+M_0^X
+\max_{s\in[-\bar\tau,t+2\bar\tau]}d_X(s)}
\kappa(r)\,dr
\le\mathcal W(2\bar\tau),
\qquad t\ge2\bar\tau.
\]
Letting $t\to+\infty$ and using monotone convergence yields
\begin{equation}\label{lim2}
\frac{\eta e^{-4K\bar\tau}}{3}\int_0^{
\bar\tau C_0^V+M_0^X
+\sup_{s\in[-\bar\tau,+\infty)}d_X(s)}
\kappa(r)\,dr
\le\mathcal W(2\bar\tau).
\end{equation}
Finally, by \eqref{infint}, \eqref{lim2} implies that there exists a positive
constant $d^*$ such that
\begin{equation}\label{firstcond}
\bar\tau C_0^V+M_0^X
+\sup_{s\in[-\bar\tau,+\infty)}d_X(s)
\le d^*.
\end{equation}
Thanks to this uniform bound for the space diameter, we have that, for all $t \geq 0$
		$$\Lambda(t) \geq \Lambda_*, \quad \mbox{where} \quad  \Lambda_*:= \kappa(d^*) = \min_{\sigma\in [0,d^*]}\left\{\psi(\sigma), \phi(\sigma), \rho(\sigma)\right\}.$$
We can write $ 1-\Gamma_n \leq 1-\xi,$
where
$$ \xi:=\Lambda_*\eta e^{-4K\bar \tau},$$
obtaining, from  \eqref{estimate_diam1} 
		\begin{equation}\label{decunif}
			D_{n+1}\leq (1-\xi)D_{n-2}.
		\end{equation}
		Thus, using the non-increasing property of $\left(D_n\right)_{n \in \mathbb{N}_0},$ thanks to an induction argument, we can write
		\begin{equation}\label{decunif2}
			D_{3n}\leq (1-\xi)^n D_0,\quad \forall n\in \mathbb{N}_0.
		\end{equation}
		As a consequence, inequality \eqref{decunif2} can be rewritten as
		\begin{equation}\label{finalmente}
			D_{3n}\leq e^{-3n\gamma \bar\tau}D_0,\quad \forall n\in\mathbb{N}_0,
		\end{equation}
		where $$ \gamma = \frac{1}{3\bar\tau}\ln \left(\frac{1}{1-\xi}\right) $$ is the positive constant.
		\\ 
		For every fixed $t\geq 0$, there exists $n\in\mathbb N_0$
such that $t\in[3n\bar\tau-\bar\tau,\,3n\bar\tau+2\bar\tau]$. Therefore, using \eqref{finalmente} and Proposition \ref{prop:preliminary} \emph{(ii)}, for $t \leq 3n\bar\tau+2\bar\tau,$
		$$d_{V}(t)\leq D_{3n}\leq e^{-3n\gamma\bar\tau}D_0 \leq e^{-\gamma(t-2\bar\tau)}D_0.$$
			\end{proof}
	
\section{Global existence of measure-valued solutions of the mean-field models}\label{wellp}
In this section, we establish the global-in-time existence and uniqueness of measure-valued solutions to the mean-field systems \eqref{pde1} and \eqref{pde2}. We first collect the force-field estimates used in both regimes, and then construct the corresponding solutions.

Throughout this section, the communication functions $\psi,\phi,\rho$ are positive, bounded by $K$, and Lipschitz continuous. We denote their Lipschitz constants by $L_\psi,L_\phi,L_\rho$ and set
\[
    L:=\max\{L_\psi,L_\phi,L_\rho\}.
\]
\begin{lem}[Force-field estimates]\label{lip}\label{lip2}
Let $T>0$. Consider either of the following settings:
\begin{itemize}
\item \emph{Case (i).} Let
$f\in C([-\bar\tau,T);\mathcal P_1(\mathbb R^{2d}))$ satisfy,
for some $R>0$,
\[
    \operatorname{supp}f_s\subseteq\overline B^{2d}(0,R),
    \qquad s\in[-\bar\tau,T),
\]
and let $\{(y_j,w_j)\}_{j=1}^m$ solve \eqref{leaders}, with
$|w_j(s)|\leq C_0^V$ for $s\in[-\bar\tau,T)$, as in Proposition~\ref{prop:preliminary} \emph{(iii)}.
The follower force is $F^m[f]$, defined by \eqref{flux1}.

\item \emph{Case (ii).} Let
$g,f\in C([-\bar\tau,T);\mathcal P_1(\mathbb R^{2d}))$ satisfy,
for some $R_1,R_2>0$,
\[
    \operatorname{supp}g_s\subseteq\overline B^{2d}(0,R_1),
    \qquad
    \operatorname{supp}f_s\subseteq\overline B^{2d}(0,R_2),
    \qquad s\in[-\bar\tau,T).
\]
The leader and follower forces are
$F_{\mathrm L}[g]$ and $F_{\mathrm F}[f,g]$, defined by
\eqref{flux2}--\eqref{flux3}.
\end{itemize}
Each force field satisfies the following estimates in its corresponding setting, for $t\in[0,T)$.

\begin{itemize}
\item[(i)] \emph{Local Lipschitz continuity.}
In Case (i),
\begin{equation}\label{diff_flux}
    |F^m[f](t,x,v)-F^m[f](t,\tilde x,\tilde v)|
    \leq\tilde K(|x-\tilde x|+|v-\tilde v|),
\end{equation}
where
\[
    \tilde K:=\max\{2RL_\phi+(C_0^V+R)L_\rho,\,2K\}.
\]
In Case (ii),
\begin{align*}
    |F_{\mathrm L}[g](t,y,w)
       -F_{\mathrm L}[g](t,\tilde y,\tilde w)|
    &\leq K_1(|y-\tilde y|+|w-\tilde w|),\\
    |F_{\mathrm F}[f,g](t,x,v)
       -F_{\mathrm F}[f,g](t,\tilde x,\tilde v)|
    &\leq K_2(|x-\tilde x|+|v-\tilde v|),
\end{align*}
where
\[
    K_1:=\max\{K,\,2R_1L\},
    \qquad
    K_2:=\max\{2K,\,L(R_1+3R_2)\}.
\]
The points in \eqref{diff_flux} belong to
$\overline B^{2d}(0,R)$; the leader and follower points
in Case (ii) belong to $\overline B^{2d}(0,R_1)$ and
$\overline B^{2d}(0,R_2)$, respectively.

\item[(ii)] \emph{Boundedness on the same balls.}
In Case (i),
\begin{equation}\label{norm_flux}
    |F^m[f](t,x,v)|\leq\tilde C,
    \qquad \tilde C:=K(3R+C_0^V).
\end{equation}
In Case (ii),
\[
    |F_{\mathrm L}[g](t,y,w)|\leq C_1,
    \qquad
    |F_{\mathrm F}[f,g](t,x,v)|\leq C_2,
\]
where $C_1:=2KR_1$ and $C_2:=K(R_1+3R_2)$.

\item[(iii)] \emph{Global linear growth.}
For arbitrary points in $\mathbb R^{2d}$, in Case (i),
\begin{equation}\label{growthFm}
    |F^m[f](t,x,v)|\leq K(R+C_0^V)+2K|v|,
\end{equation}
whereas in Case (ii),
\begin{equation}\label{growth2}
\begin{aligned}
    |F_{\mathrm L}[g](t,y,w)|
    &\leq K(R_1+|w|),\\
    |F_{\mathrm F}[f,g](t,x,v)|
    &\leq K(R_1+R_2)+2K|v|.
\end{aligned}
\end{equation}

\item[(iv)] \emph{Dependence on the measure histories.}
Let $\widetilde f$ and, in Case (ii), $\widetilde g$ be curves
in $\mathcal P_1(\mathbb R^{2d})$, continuous on $[-\bar\tau,T)$
with respect to $d_1$, and satisfying the same support bounds
as the corresponding curves $f,g$.
In Case (i), with the same leader trajectories and
$C_R:=K+2RL_\phi$,
\begin{equation}\label{historyFm}
    \|F^m[f](t,\cdot)-F^m[\widetilde f](t,\cdot)
       \|_{L^\infty(B^{2d}(0,R))}
    \leq C_R\int_{t-\tau_2(t)}^t
       G(t-s)d_1(f_s,\widetilde f_s)\,ds.
\end{equation}
In Case (ii), set $R:=\max\{R_1,R_2\}$ and $C_R:=K+2RL$.
Then
\begin{equation}\label{historyFL}
    \|F_{\mathrm L}[g](t,\cdot)
       -F_{\mathrm L}[\widetilde g](t,\cdot)
       \|_{L^\infty(B^{2d}(0,R_1))}
    \leq C_R\int_{t-\tau_1(t)}^t
       G(t-s)d_1(g_s,\widetilde g_s)\,ds,
\end{equation}
and
\begin{equation}\label{historyFF}
\begin{aligned}
    &\|F_{\mathrm F}[f,g](t,\cdot)
       -F_{\mathrm F}[\widetilde f,\widetilde g](t,\cdot)
       \|_{L^\infty(B^{2d}(0,R_2))}\\
    &\qquad\leq C_R\int_{t-\tau_2(t)}^t
       G(t-s)d_1(f_s,\widetilde f_s)\,ds +C_R\int_{t-\tau_1(t)}^t
       G(t-s)d_1(g_s,\widetilde g_s)\,ds.
\end{aligned}
\end{equation}
\end{itemize}
All constants depend only on the stated support or velocity bounds,
$K$, and the Lipschitz constants of the communication functions.
\end{lem}

\begin{proof}
All three interactions (leader-leader, follower-follower, leader-follower) are controlled by the same elementary estimate.
For $\chi\in\{\psi,\phi,\rho\}$, write
\begin{align*}
    &\chi(|x-\bar x|)(\bar v-v)
       -\chi(|x'-\bar x'|)(\bar v'-v')\\
    &\quad=
       [\chi(|x-\bar x|)-\chi(|x'-\bar x'|)](\bar v-v)\\
    &\qquad\quad+
       \chi(|x'-\bar x'|)
       [(\bar v-v)-(\bar v'-v')].
\end{align*}
The Lipschitz continuity and boundedness of $\chi$ give
\begin{equation}\label{interaction_difference}
\begin{aligned}
    &|\chi(|x-\bar x|)(\bar v-v)
       -\chi(|x'-\bar x'|)(\bar v'-v')|\\
    &\quad\leq
       L_\chi(|\bar v|+|v|)
       (|x-x'|+|\bar x-\bar x'|) +
       K(|v-v'|+|\bar v-\bar v'|).
\end{aligned}
\end{equation}
In particular, when $|v|,|\bar v|\leq R$, the first coefficient is
bounded by $2RL_\chi$. To compare evaluation points while keeping the histories fixed, take
$(\bar x',\bar v')=(\bar x,\bar v)$ in
\eqref{interaction_difference}. Since each measure has unit mass
and each memory window has mass at most one, integration against
$f_s$ and averaging over the leaders yield
$$
  |F^m[f](t,x,v)-F^m[f](t,\tilde x,\tilde v)| \leq
       [2RL_\phi+(C_0^V+R)L_\rho]|x-\tilde x|
       +2K|v-\tilde v|.
$$
This proves \eqref{diff_flux}. The same estimate applied to the
leader field gives
\[
    |F_{\mathrm L}[g](t,y,w)
       -F_{\mathrm L}[g](t,\tilde y,\tilde w)|
    \leq 2R_1L_\psi|y-\tilde y|+K|w-\tilde w|.
\]
For the follower field in Case (ii), the follower and leader
contributions give, respectively, the spatial coefficients
$2R_2L_\phi$ and $(R_1+R_2)L_\rho$. Thus
$$|F_{\mathrm F}[f,g](t,x,v)
       -F_{\mathrm F}[f,g](t,\tilde x,\tilde v)|\leq
       [2R_2L_\phi+(R_1+R_2)L_\rho]|x-\tilde x|
       +2K|v-\tilde v|.
$$
Using $L_\psi,L_\phi,L_\rho\leq L$ proves the asserted estimates
with $K_1$ and $K_2$. \\
For arbitrary evaluation points, direct integration of
$|\chi(|x-\bar x|)(\bar v-v)|\leq K(|\bar v|+|v|)$ gives
\begin{align*}
    |F^m[f](t,x,v)|
    &\leq K(R+|v|)+K(C_0^V+|v|),\\
    |F_{\mathrm L}[g](t,y,w)|
    &\leq K(R_1+|w|),\\
    |F_{\mathrm F}[f,g](t,x,v)|
    &\leq K(R_2+|v|)+K(R_1+|v|).
\end{align*}
These are \eqref{growthFm} and \eqref{growth2}. Restricting the
evaluation velocities to their corresponding balls also proves
\eqref{norm_flux} and the bounds with $C_1,C_2$.
Consequently, the complete characteristic vector fields have
at most linear growth. \\
Finally, fix an evaluation point $(x,v)$ with $|v|\leq R$,
and let $q,\widetilde q$ be probability measures supported in
$\overline B^{2d}(0,R)$. For any transport plan
$\pi\in\Pi(q,\widetilde q)$, its marginal properties and
\eqref{interaction_difference}, with $(x',v')=(x,v)$, imply
\begin{align*}
    &\left|
       \int_{\mathbb R^{2d}}\chi(|x-\bar x|)(\bar v-v)
           \,q(d\bar x,d\bar v)
       -\int_{\mathbb R^{2d}}\chi(|x-\bar x|)(\bar v-v)
           \,\widetilde q(d\bar x,d\bar v)
       \right|\\
    &\quad\leq
       \int_{\mathbb R^{2d}\times\mathbb R^{2d}}
       \left(2RL_\chi|\bar x-\bar x'|
             +K|\bar v-\bar v'|\right)
       \,\pi(d\bar x,d\bar v,d\bar x',d\bar v')\\
    &\quad\leq
       (K+2RL_\chi)
       \int_{\mathbb R^{2d}\times\mathbb R^{2d}}
       |(\bar x,\bar v)-(\bar x',\bar v')|\,d\pi.
\end{align*}
Taking the infimum over transport plans bounds this expression by
$(K+2RL_\chi)d_1(q,\widetilde q)$. \\
 In Case (i), the leader
contribution cancels because the trajectories are fixed.
Applying the preceding bound with $\chi=\phi$ and integrating
over $[t-\tau_2(t),t]$ proves \eqref{historyFm}.
In Case (ii), apply it with $\chi=\psi$ to the leader force,
and with $\chi=\phi,\rho$ to the two follower contributions.
Integration over the corresponding memory windows and taking
the supremum over the prescribed evaluation balls prove
\eqref{historyFL} and \eqref{historyFF}.
\end{proof}

\subsection{Well-posedness: Few leaders and many followers system}
We first solve the leader dynamics \eqref{leaders} and then construct
the follower solution of \eqref{pde1}, with force $F^m[f]$ defined
by \eqref{flux1}. The force-field estimates of Lemma~\ref{lip}
provide the regularity, growth, and history dependence needed
for this construction.

\begin{proof}[Proof of Theorem \ref{limit_result}: existence and uniqueness in Case (i)]
We begin by considering the ODE system describing the evolution of the leader trajectories $\{y_j(t), w_j(t)\}_{j=1}^m$. Since the interaction functions are Lipschitz continuous and bounded, standard results from the theory of delay differential equations (see, e.g., \cite{Halanay, Hale}) ensure existence and uniqueness of solutions. Specifically, by applying the Banach fixed-point theorem on small time intervals and iterating the solution step-by-step, we can construct a unique global-in-time solution for the leader dynamics.

We now turn to the second component of the system, equation \eqref{pde1}, which is a Vlasov-type kinetic equation driven by a delayed,
nonlocal force field. To construct a local-in-time measure-valued solution, we adapt the fixed-point argument in the proof of \cite[Theorem~3.10]{Carrillo} to the dependence of the force field on the measure history. Lemma~\ref{lip} provides the local regularity and linear-growth estimates for prescribed measure curves, while \eqref{historyFm} controls the dependence on their histories. Thus, to extend this local-in-time solution to a global one, it is sufficient to control the growth of the support of $f_t$ in both position and velocity. We do this by estimating the maximal spatial extension of the support. \\
By Proposition~\ref{prop:preliminary} \emph{(iii)}, we have the uniform-in-time bound
\[
|w_j(t)| \leq C_0^y, \quad \text{for all } t \geq 0 \, \text{ and } j = 1,\dots,m,
\]
where 
$$ C_0^y:= \max_{s \in [-\bar\tau,0]}\max_{j=1,\dots,m}|w_j(s)|.$$
Let us define the maximal radius of the support of the measure \( f_t \) with respect to the velocity by
\begin{equation}\label{RV}
R_V(t) := \sup_{s \in [-\bar\tau,t]} \left\{ \sup_{v \in \overline{\supp_V f_s}} |v|, \;  C_0^y \right\}.
\end{equation}
Analogously, the maximal radius of the support of $f_t$ with respect to space is defined by 
\begin{equation}\label{RX}
R_X(t) := \sup_{s \in [-\bar\tau,t]} \left\{ \sup_{x \in \overline{\supp_X f_s}} |x|, \;  \max_{j=1,\dots,m} \vert y_j(s)\vert \right\}.
\end{equation}
We first construct a local solution on a time interval $[0,h]$,
where $h>0$ will be chosen sufficiently small. Let $R_0$ be the
support bound for the initial histories and choose $R>R_0$.
Consider the space
$$
    \mathcal X_{h,R}:=
    \{q\in C([-\bar\tau,h];\mathcal P_1(\mathbb R^{2d})):
    \ q_s=\nu_s, \ s\in[-\bar\tau,0], \ \supp q_s\subseteq\overline B^{2d}(0,R)
       \ s\in[-\bar\tau,h]\},
$$
equipped with the distance
\[
    d_h(q,r):=\sup_{0\le t\le h}d_1(q_t,r_t).
\]
This is a complete metric space. For each prescribed curve $q\in\mathcal X_{h,R}$, consider
the characteristic system
\[
    \begin{cases}
    \displaystyle\frac{d}{dt}X_q(t;x,v)=V_q(t;x,v),\\[1mm]
    \displaystyle\frac{d}{dt}V_q(t;x,v)
    =F^m[q](t,X_q(t;x,v),V_q(t;x,v)),
    \end{cases}
\]
with initial conditions
\[
    X_q(0;x,v)=x,\qquad V_q(0;x,v)=v.
\]
By Lemma~\ref{lip} and, in particular, \eqref{growthFm}, this system admits
a unique solution on $[0,h]$ for every $(x,v)\in\mathbb R^{2d}$.
Write $Z_q(t;x,v):=(X_q(t;x,v),V_q(t;x,v))$.
Define the map $\mathcal T$ by
\[
    (\mathcal Tq)_t:=
    \begin{cases}
    \nu_t, & t\in[-\bar\tau,0],\\
    Z_q(t;\cdot)_{\#}\nu_0, & t\in(0,h].
    \end{cases}
\]
On $\overline B^{2d}(0,R)$, the complete characteristic
vector field satisfies
\[
    |(v,F^m[q](t,x,v))|
    \le M_R:=R+K(3R+C_0^V),
\]
uniformly in $q\in\mathcal X_{h,R}$.
Choosing $hM_R<R-R_0$, a first-exit argument shows that
characteristics starting in $\supp\nu_0$ remain in
$\overline B^{2d}(0,R)$ on $[0,h]$.
The transported curve is continuous in $d_1$.
Thus $\mathcal T$ maps $\mathcal X_{h,R}$ into itself. \\
Let $\ell_R$ be a uniform Lipschitz constant of the complete
characteristic vector fields on $\overline B^{2d}(0, R)$.
For $q,r\in\mathcal X_{h, R}$, their negative-time histories
coincide. Hence \eqref{historyFm} and the memory-kernel
bound imply
\[
    \|F^m[q](t,\cdot)-F^m[r](t,\cdot)\|_{L^\infty(B^{2d}(0,R))}
    \le C_Rd_h(q,r).
\]
Comparing characteristics with the same initial point
and applying Gronwall's inequality gives
\[
    \sup_{z\in\supp\nu_0}|Z_q(t;z)-Z_r(t;z)|
    \le C_Rt e^{\ell_Rt}d_h(q,r).
\]
Since $(Z_q(t;\cdot),Z_r(t;\cdot))_{\#}\nu_0$ is a
transport plan between $(\mathcal Tq)_t$ and $(\mathcal Tr)_t,$
we obtain
$$
    d_h(\mathcal{T}q,\mathcal{T}r)  \le C_R h e^{\ell_R h}d_h(q,r).
$$
Choosing $h$ further so that $ C_R h e^{\ell_R h}<1$,
Banach's fixed-point theorem yields a unique fixed point
$f\in\mathcal X_{h,R}$, satisfying
\[
    f_t=Z_f(t;\cdot)_{\#}\nu_0,\qquad t\in[0,h].
\]
The characteristic chain rule gives the weak formulation.
Uniqueness among compactly supported weak solutions follows
from uniqueness of linear transport for a prescribed field
and the same contraction argument on a common support ball.
From now on, we write $X, V$ for the characteristics
associated with this fixed point. \\
 We now perform a continuity argument to control $R_V(t)$. For a fixed $\epsilon > 0$ we define a set
\[
\mathcal S^\epsilon:= \left\{ t\in(0,h]:  R_V(s)<R_V(0)+\epsilon\text{ for every }s\in[0,t)\right\}.
\]
By continuity of trajectories,  $\mathcal{S}^{\epsilon}$ is nonempty. Let $T_{\epsilon}:= \sup \mathcal{S}^{\epsilon}$. Our goal is to show that  $T_\epsilon = h.$ Suppose, by contradiction, that $T_{\epsilon} < h$. Then, we find
\begin{equation}\label{A}
\lim_{t \rightarrow T_{\epsilon}^-} R_V(t) = R_V(0)+\epsilon
\end{equation}
and 
\[%\begin{equation}\label{B}
R_V(t) < R_V(0) + \epsilon, \quad \ \forall \ t < T_{\epsilon}.
\]%\end{equation}
Fix $(x,v)\in\supp\nu_0$. Let us simplify notation by writing $X(t;x,v)$ as $X(t)$ and $V(t;x,v)$ as $V(t).$ 
 Then, we estimate
\begin{equation*}
\begin{split}
& \frac{1}{2} \frac{d}{d t} \vert V(t)\vert^2   = \langle \dot{V}(t), V(t) \rangle = \langle F^m[f](t,X(t),V(t)), V(t) \rangle  \\
& = \int_{t-\tau_2(t)}^t G(t-s) \int_{\mathbb{R}^{2d}} \phi(\vert X(t)-\bar{x}|)\langle \bar{v}-V(t),V(t)\rangle f_{s}(d\bar{x}, d \bar{v}) \ ds  \\
& \hspace{2cm} +\frac{1}{m} \sum_{j=1}^m \int_{t-\tau_1(t)}^t G(t-s) \rho(\vert X(t)-y_j(s)\vert ) \langle w_j(s)-V(t),V(t) \rangle ds  \\
& = \int_{t-\tau_2(t)}^t G(t-s)
    \int_{\mathbb{R}^{2d}} \phi(\vert X(t)-\bar{x}\vert)
    \langle \bar{v},V(t)\rangle f_s(d\bar{x},d\bar{v})\,ds\\
& \hspace{2cm}
    -\int_{t-\tau_2(t)}^t G(t-s)
    \int_{\mathbb{R}^{2d}} \phi(\vert X(t)-\bar{x}\vert)
    |V(t)|^2 f_s(d\bar{x},d\bar{v})\,ds\\
& \hspace{2cm} + \frac{1}{m} \sum_{j=1}^m \int_{t-\tau_1(t)}^t G(t-s) \rho(\vert X(t)-y_j(s)\vert ) \left( \langle w_j(s),V(t) \rangle - \vert V(t) \vert^2 \right) \ ds.
\end{split}
\end{equation*}
Using the definition \eqref{RV}, we deduce
\begin{align*}
\frac{1}{2} \frac{d}{d t} \vert V(t)\vert^2 \leq & \ \vert V(t)\vert \int_{t-\tau_2(t)}^t G(t-s) \int_{\mathbb{R}^{2d}} \phi(\vert X(t)-\bar{x}\vert) (R_V(t)- \vert V(t) \vert) f_{s}(d\bar{x}, d \bar{v}) \ ds \cr
&+ \vert V(t)\vert \ \frac{1}{m} \sum_{j=1}^m \int_{t-\tau_1(t)}^t G(t-s) \rho(\vert X(t)-y_j(s)\vert )(R_V(t)-\vert V(t)\vert) \ ds.
\end{align*}
Since $R_V(t) - |V(t)| \geq 0$ for $t \in (0,T_{\epsilon})$, and $\phi$ and $\rho$ are bounded by \( K \), we find
\[
\frac{d}{dt} |V(t)| \leq 4K (R_V(0) + \epsilon - |V(t)|).
\]
By Gronwall's inequality, for every $(x,v)\in\supp\nu_0$
and $t\in[0,T_\epsilon]$, we obtain
$$
    |V(t;x,v)|
    \leq R_V(0)+\epsilon
    -\bigl(R_V(0)+\epsilon-|v|\bigr)e^{-4Kt}\leq R_V(0)+\epsilon-\epsilon e^{-4Kt}.
$$
Taking the supremum over these characteristics and all
earlier times, and recalling the bound on the initial
history and the leader velocities, yields
\[
    R_V(T_\epsilon)
    \leq R_V(0)+\epsilon-\epsilon e^{-4KT_\epsilon}
    <R_V(0)+\epsilon.
\]
This contradicts \eqref{A}. Therefore $T_\epsilon=h$.
Since $\epsilon>0$ was arbitrary, we conclude that
\[
    R_V(t)\leq R_V(0),\qquad t\in[0,h].
\]
Moreover, integrating the position equations gives
\[
    R_X(t)\leq R_X(0)+tR_V(0),\qquad t\in[0,h].
\]

The local fixed-point construction can be restarted at
any time $t_0$, with the already determined measure history
prescribed on $[t_0-\bar\tau,t_0]$.
The same argument propagates the velocity bound and the
linear-in-time position bound.
Consequently, the supports remain in a common bounded ball
on every finite time interval, and an admissible local
time step can be chosen uniformly on each finite horizon.
Repeating the construction therefore yields a unique
global-in-time solution. Finally, the characteristic representation obtained in
the fixed-point construction yields the weak formulation
\eqref{measure}, as in \cite{Carrillo}, and
\[
    f_t=(X(t;\cdot),V(t;\cdot))_{\#}\nu_0,
    \qquad t\geq0.
\]
\end{proof}

\subsection{Well-posedness: Infinite population limit for both leaders and followers}
For the fully kinetic system \eqref{pde2}, we construct the leader
solution first and then the follower solution. The fixed-point
construction and continuation argument are those of Case (i);
below we specify the force fields, initial histories, and support
bounds for the two populations.

\begin{proof}[Proof of Theorem \ref{limit_result}: existence and uniqueness in Case (ii)]
We use the fixed-point construction of Case (i),
which adapts the argument in
\cite{Carrillo} to distributed delays. \\
First, we construct the leader solution with prescribed
history $g_s=\bar\mu_s$ for $s\in[-\bar\tau,0]$.
For a candidate curve $q$, let $Z_{\mathrm L}^q$ denote
the characteristic flow generated by
$(w,F_{\mathrm L}[q])$. The map
\[
    q\longmapsto
    Z_{\mathrm L}^q(t;\cdot)_{\#}\bar\mu_0,
    \qquad t\geq0,
\]
retaining the prescribed history at nonpositive times,
is defined on the same type of complete metric space
as in Case (i).
The local regularity and growth estimates of Lemma~\ref{lip} give the
assigned characteristic flows and support preservation
on a sufficiently short interval.
Estimate \eqref{historyFL} yields the contraction bound
\[
    d_h(\mathcal T_{\mathrm L}q,\mathcal T_{\mathrm L}r)
    \leq C_Rh e^{\ell_Rh}d_h(q,r).
\]
Hence Banach's theorem gives a unique local leader solution.
The velocity bound and the linear-in-time position bound
follow by the same barrier argument as in Case (i),
and allow this solution to extend globally. \\
We next keep this leader solution $g$ fixed and construct
the follower solution with prescribed history
$f_s=\bar\nu_s$ for $s\in[-\bar\tau,0]$.
When comparing two follower curves, the leader contribution
cancels. Thus \eqref{historyFF}, with
$\widetilde g=g$, gives precisely the history estimate
required for the fixed-point construction of Case (i).
Together with \eqref{growth2} and the local regularity
estimates, this yields a unique local follower solution.

We now define
\[
    \widetilde R_X(t):=
    \sup_{-\bar\tau\leq s\leq t}
    \max\left\{
        \sup_{x\in\supp_X g_s}|x|,
        \sup_{x\in\supp_X f_s}|x|
    \right\},
\]
and
\[
    \widetilde R_V(t):=
    \sup_{-\bar\tau\leq s\leq t}
    \max\left\{
        \sup_{v\in\supp_V g_s}|v|,
        \sup_{v\in\supp_V f_s}|v|
    \right\}.
\]
The leader velocities are already bounded by their
initial-history maximum.
Applying the barrier argument of Case (i) to the follower
characteristics therefore gives
\[
    \widetilde R_V(t)\leq\widetilde R_V(0).
\]
Integrating the position equations for both populations
then yields
\[
    \widetilde R_X(t)
    \leq\widetilde R_X(0)+t\widetilde R_V(0).
\]
The follower fixed-point construction can be restarted
at any time $t_0$, with its already determined history
prescribed on $[t_0-\bar\tau,t_0]$ and with $g$ fixed.
The preceding bounds keep both supports in a common
bounded ball on every finite time interval.
Hence an admissible local step can be chosen uniformly
on each finite horizon, yielding a unique global
follower solution. \\
Both solutions satisfy the weak formulation by the
characteristic chain rule, and
\[
    g_t=Z_{\mathrm L}^g(t;\cdot)_{\#}\bar\mu_0,
    \qquad
    f_t=Z_{\mathrm F}^{f,g}(t;\cdot)_{\#}\bar\nu_0,
\]
where the two flows are associated with
$(w,F_{\mathrm L}[g])$ and $(v,F_{\mathrm F}[f,g])$,
respectively.
Thus, for every finite $T>0$, both curves belong to
$C([-\bar\tau,T];\mathcal P_1(\mathbb R^{2d}))$
and have uniformly compact supports on that interval.
\end{proof}

\section{Stability and flocking in the mean-field regime}\label{stability}
In this section, we establish stability with respect to the initial histories and then transfer the particle flocking estimates to the mean-field regime. The arguments are presented simultaneously for the mixed system \eqref{leaders}--\eqref{pde1} and the fully mean-field system \eqref{pde2}.

\subsection{Wasserstein stability estimate}
Throughout this subsection, the measure-valued histories are continuous with values in $\mathcal P_1(\mathbb R^{2d})$ and have uniformly compact supports. The interaction functions $\psi,\phi,\rho$ are bounded and Lipschitz continuous. No flocking assumption is needed for the stability estimate. We use the phase-space norm
\[
\lVert(x,v)-(\tilde x,\tilde v)\rVert_*
=|x-\tilde x|+|v-\tilde v|
\]
and its associated $1$-Wasserstein distance $d$. The same conclusions hold for the Euclidean cost, with a modified constant, since
$|z|_2\leq\lVert z\rVert_*\leq\sqrt2|z|_2$. \\
For a fixed $T>0$, let $R$ be a common bound for all velocities of the solutions being compared, including both populations and their histories on $[-\bar\tau,T]$. 
Thanks to \eqref{memory_cond}, the upper bound on the memory masses is written as
\begin{equation}\label{stability_memory_upper}
G\geq0,\qquad
b_i(t):=\int_{t-\tau_i(t)}^tG(t-s)\,ds\leq B,
\qquad i=1,2,\quad t\in[0,T].
\end{equation}
For example, one may take $B=\|G\|_{L^1(0,\bar\tau)}\leq 1.$

\begin{lem}\label{4.3}\label{stab}
Consider two solutions of either mean-field system. Assume that they satisfy the hypotheses of Theorem~\ref{limit_result}. Under the preceding bounds, there exists $C_T>0$, depending only on $T,R,K,L,B$, such that the following estimates hold for every $t\in[0,T]$.

\smallskip
\noindent\textit{\emph{(i) Finitely many leaders.}} For solutions $\{(y_j^i,w_j^i),f_t^i\}$, $i=1,2$, with histories $\{(y_j^{i,0},w_j^{i,0}),\nu_s^i\}$,
\begin{equation}\label{stability_mixed}
\begin{split}
&d(f_t^1,f_t^2)+\frac1m\sum_{j=1}^m
\lVert(y_j^1,w_j^1)(t)-(y_j^2,w_j^2)(t)\rVert_*\\
&\quad\leq C_T\sup_{s\in[-\bar\tau,0]}d(\nu_s^1,\nu_s^2) +C_T\sup_{s\in[-\bar\tau,0]}\frac1m\sum_{j=1}^m
\lVert(y_j^{1,0},w_j^{1,0})(s)-(y_j^{2,0},w_j^{2,0})(s)\rVert_*.
\end{split}
\end{equation}

\noindent\textit{\emph{(ii) Two measure-valued populations.}} For solutions $(g_t^i,f_t^i)$, $i=1,2$, with histories $(\bar\mu_s^i,\bar\nu_s^i)$,
\begin{equation}\label{stability_kinetic}
d(g_t^1,g_t^2)+d(f_t^1,f_t^2)
\leq C_T\sup_{s\in[-\bar\tau,0]}d(\bar\mu_s^1,\bar\mu_s^2) +C_T\sup_{s\in[-\bar\tau,0]}d(\bar\nu_s^1,\bar\nu_s^2).
\end{equation}
In particular, the constant is independent of population sizes.
\end{lem}

\begin{proof}
We represent the leaders by $g_t^i$ in both regimes. In Case (i), this means
\[
g_t^i=\frac1m\sum_{j=1}^m\delta_{(y_j^i,w_j^i)(t)},
\qquad
\mu_s^i=\frac1m\sum_{j=1}^m\delta_{(y_j^{i,0},w_j^{i,0})(s)}
\quad s\in [-\bar \tau,0].
\]
In Case (ii), set $\mu_s^i=\bar\mu_s^i$ and, for the common proof, write $\nu_s^i=\bar\nu_s^i$. The empirical interpretation uses the same normalized interaction sums as the corresponding kinetic fields. \\
Let $Z^i$ and $Y^i$ denote, respectively, the follower and leader characteristic flows in phase space. Their position components have derivatives equal to their velocity components, and their velocity components solve the characteristic equations associated with the corresponding interaction fields. The characteristic representation gives
\[
f_t^i=Z^i(t;\cdot)_\#\nu_0^i,
\qquad g_t^i=Y^i(t;\cdot)_\#\mu_0^i.
\]
In the discrete case, $Y^i$ sends each initial leader state to the corresponding leader trajectory. \\
Choose an optimal coupling $\pi_0\in\Pi(\nu_0^1,\nu_0^2)$. For the leaders, take the coupling that matches equal labels in Case (i),
\[
\eta_0=\frac1m\sum_{j=1}^m
\delta_{((y_j^{1,0},w_j^{1,0})(0),(y_j^{2,0},w_j^{2,0})(0))},
\]
and an optimal coupling $\eta_0\in\Pi(\bar\mu_0^1,\bar\mu_0^2)$ in Case (ii). Define the transported plans
\begin{equation}\label{T}
\begin{split}
\pi_t&=(Z^1(t;\cdot),Z^2(t;\cdot))_\#\pi_0\in\Pi(f_t^1,f_t^2),\\
\eta_t&=(Y^1(t;\cdot),Y^2(t;\cdot))_\#\eta_0\in\Pi(g_t^1,g_t^2),
\end{split}
\end{equation}
and their costs
\[
\theta(t)=\int\lVert z-z'\rVert_*\,\pi_t(dz,dz'),
\qquad
\xi(t)=\int\lVert z-z'\rVert_*\,\eta_t(dz,dz').
\]
Thus $d(f_t^1,f_t^2)\leq\theta(t)$ and $d(g_t^1,g_t^2)\leq\xi(t)$. Moreover, in Case (i),
\[
\xi(t)=\frac1m\sum_{j=1}^m
\lVert(y_j^1,w_j^1)(t)-(y_j^2,w_j^2)(t)\rVert_*.
\]
For $s\in[-\bar\tau,0]$, extend $\theta$ by $\theta(s)=d(\nu_s^1,\nu_s^2)$, and extend $\xi$ by the same labelled average of the leader histories in Case (i), i.e. 
\[
\xi(s)=\frac1m\sum_{j=1}^m
\lVert(y_j^{1,0},w_j^{1,0})(s)-(y_j^{2,0},w_j^{2,0})(s)\rVert_*,
\]
or by $d(\bar\mu_s^1,\bar\mu_s^2)$ in Case (ii). These definitions agree at zero. No optimal transport map is required. \\
We now estimate the interactions once for both populations. If $ \chi \in\{\psi,\phi,\rho\}$ and the velocities involved have modulus at most $R$, using \eqref{interaction_difference}, then
\begin{equation}\label{interaction_stability}
| \chi (|x-\bar x|)(\bar v-v)-\chi(|x'-\bar x'|)(\bar v'-v')| \leq(2RL_\chi+K)
\bigl(\lVert z-z'\rVert_*+\lVert\bar z-\bar z'\rVert_*\bigr),
\end{equation}
where $z=(x,v)$ and $\bar z=(\bar x,\bar v)$, and similarly for the primed variables. \\
Set $C_1=2RL_\phi+K$, $C_2=2RL_\rho+K$, and
$C_3=2RL_\psi+K$. We estimate the differences between the 
follower-follower and leader-follower interaction terms separately.
Define
\begin{align*}
I := {}&
\Bigg|
\int_{t-\tau_2(t)}^t G(t-s)
\Bigg[
\int_{\mathbb R^{2d}}
\phi(|x-\bar x|)(\bar v-v)\,
f_s^1(d\bar x,d\bar v)\\
&\hspace{3cm}
-
\int_{\mathbb R^{2d}}
\phi(|x'-\bar x|)(\bar v-v')\,
f_s^2(d\bar x,d\bar v)
\Bigg]\,ds
\Bigg|,\\
II := {}&
\Bigg|
\int_{t-\tau_1(t)}^t G(t-s)
\Bigg[
\int_{\mathbb R^{2d}}
\rho(|x-\bar y|)(\bar w-v)\,
g_s^1(d\bar y,d\bar w)\\
&\hspace{3cm}
-
\int_{\mathbb R^{2d}}
\rho(|x'-\bar y|)(\bar w-v')\,
g_s^2(d\bar y,d\bar w)
\Bigg]\,ds
\Bigg|.
\end{align*}
By the triangle inequality, the difference between the
two follower force fields, evaluated at $z$ and $z'$,
is bounded by $I+II$.
Applying \eqref{interaction_stability} to the integrands
and using the couplings defining $\theta(s)$ and $\xi(s)$
for $s\geq0$, we obtain
\begin{align*}
I&\leq C_1\left(
b_2(t)\lVert z-z'\rVert_*
+\int_{t-\tau_2(t)}^tG(t-s)\theta(s)\,ds
\right),\\
II&\leq C_2\left(
b_1(t)\lVert z-z'\rVert_*
+\int_{t-\tau_1(t)}^tG(t-s)\xi(s)\,ds
\right).
\end{align*}
For nonnegative memory times we may use the transported plans, while for negative times we first obtain the pointwise force bound using optimal couplings, or labelled couplings for discrete leaders, and then integrate the resulting scalar estimate. In particular, no measurable selection of optimal couplings is needed. The same argument with $\chi=\psi$ bounds the leader force difference by
\[
C_3\left(b_1(t)\lVert z-z'\rVert_*
+\int_{t-\tau_1(t)}^tG(t-s)\xi(s)\,ds\right).
\]
Integrating these bounds against the current couplings and including the position equations gives, for almost every $t\in(0,T)$,
\begin{equation}\label{theta_stability_corrected}
\begin{split}
\theta'(t)\leq{}&[1+B(C_1+C_2)]\theta(t)\\
&+C_1\int_{t-\tau_2(t)}^tG(t-s)\theta(s)\,ds
+C_2\int_{t-\tau_1(t)}^tG(t-s)\xi(s)\,ds,
\end{split}
\end{equation}
and
\begin{equation}\label{xi_stability_corrected}
\xi'(t)\leq(1+BC_3)\xi(t)
+C_3\int_{t-\tau_1(t)}^tG(t-s)\xi(s)\,ds.
\end{equation}
All couplings have total mass one, so the constants are independent of population sizes. \\
To close the estimate, set
\[
H(t)=\sup_{s\in[-\bar\tau,t]}(\theta(s)+\xi(s)),
\qquad H_0=\sup_{s\in[-\bar\tau,0]}(\theta(s)+\xi(s)).
\]
Since $\tau_i(t)\leq\bar\tau$, the preceding inequalities imply
$(\theta+\xi)'(t)\leq CH(t)$ almost everywhere, with, for example,
$C=1+2B(C_1+C_2+C_3)$. For every $r\in[0,t]$,
\[
\theta(r)+\xi(r)\leq H_0+C\int_0^rH(s)\,ds
\leq H_0+C\int_0^tH(s)\,ds.
\]
Taking the supremum, including the initial interval, and applying Gr\"onwall's inequality, we obtain
\[
H(t)\leq H_0+C\int_0^tH(s)\,ds\leq H_0e^{Ct}.
\]
Finally, $H_0\leq\sup_{[-\bar\tau,0]}\theta+\sup_{[-\bar\tau,0]}\xi$. In Case (i), $\xi$ is exactly the labelled leader discrepancy; in Case (ii), it bounds the leader Wasserstein distance. Thus both \eqref{stability_mixed} and \eqref{stability_kinetic} follow with $C_T=e^{CT}$.
\end{proof}

\subsection{Mean-field limit and emergence of flocking}
We use the preceding stability estimate to pass from the particle models to either mean-field regime. Only the choice of the leader approximation differs; convergence and the passage of the flocking bounds are proved together.

\paragraph{Approximation of the initial histories.}
Write $(\mu_s,\nu_s)$ for the leader and follower histories. In Case (i),
\[
\mu_s=\frac1m\sum_{i=1}^m\delta_{(y_i^0,w_i^0)(s)},
\]
whereas in Case (ii) $(\mu_s,\nu_s)=(\bar\mu_s,\bar\nu_s)$. Take continuous particle histories with empirical measures
\[
\mu_s^m=\frac1{m}\sum_{i=1}^{m}\delta_{(y_i^{0},w_i^{0})(s)},
\qquad
\nu_s^N=\frac1{N}\sum_{j=1}^{N}\delta_{(x_j^{0},v_j^{0})(s)},
\]
such that $N\to\infty$ and
\begin{equation}\label{history_approximation}
\sup_{s\in[-\bar\tau,0]}d(\mu_s^m,\mu_s)
+\sup_{s\in[-\bar\tau,0]}d(\nu_s^N,\nu_s)\longrightarrow0.
\end{equation}
In Case (i), we keep every leader history unchanged, so that $\mu_s^m=\mu_s$. In Case (ii), let $m\to\infty$ and approximate both histories. No relation between the rates of growth of $m$ and $N$ is required. The approximations can be chosen with uniform geometric bounds. Indeed, take a fine time mesh, approximate the measure at each node by an equal-weight empirical measure with atoms in its support, and use a common number of atoms across all nodes. Match the atoms optimally at consecutive nodes and interpolate linearly. The resulting particle histories are continuous and remain in the convex hull of the union of the initial supports. If the errors at the nodes and the variation of the original history over each mesh interval are at most $\varepsilon$, the distance between consecutive empirical measures is at most $3\varepsilon$. Pairing interpolated atoms with their left endpoints then gives a uniform history error at most $5\varepsilon$. Letting $\varepsilon\downarrow0$ proves \eqref{history_approximation}. Convexification preserves the diameters of the initial position and velocity projection sets.

\paragraph{Convergence on finite time intervals.}
Let $(y_i^m,w_i^m)$ and $(x_j^N,v_j^N)$ solve the particle systems \eqref{leaders}--\eqref{followers} with these histories, and define
\[
g_t^m=\frac1{m}\sum_{i=1}^{m}\delta_{(y_i^m,w_i^m)(t)},
\qquad
f_t^N=\frac1{N}\sum_{j=1}^{N}\delta_{(x_j^N,v_j^N)(t)}.
\]
These are empirical solutions of the corresponding kinetic equations. In Case (i), the leader subsystem is unchanged, so $g_t^m=g_t$, where $g_t$ denotes the empirical measure of the fixed leader population. In Case (ii), both empirical populations vary, depending on the mesh. \\
Applying the appropriate conclusion of Lemma~\ref{4.3}, for every fixed $T>0$ we obtain the common estimate
\begin{equation}\label{empirical_stability}
\sup_{t\in[0,T]}\bigl(d(g_t^m,g_t)+d(f_t^N,f_t)\bigr) \leq C_T\left(\sup_{s\in[-\bar\tau,0]}d(\mu_s^m,\mu_s)
+\sup_{s\in[-\bar\tau,0]}d(\nu_s^N,\nu_s)\right)
\longrightarrow0.
\end{equation}
Here the leader discrepancy is identically zero in Case (i). The constants are independent of $N$ and $m$ because the initial velocity bounds are uniform.

\paragraph{Passage of the flocking estimates.}
For this step, we use the particle estimate of Theorem~\ref{uf} in the form
\begin{equation}\label{lt_dis}
d_V^n(t)\leq e^{-\gamma(t-2\bar\tau)}D_0^n,\qquad t\geq0,
\end{equation}
where $d_V^n$ denotes, in both cases, the velocity diameter
of the union of the leader and follower populations.
In Case (i), the leader population is fixed, whereas in
Case (ii) both populations are approximated. The passage to the limit requires a rate $\gamma>0$ independent of the population sizes $m$ and $N,$ and initial prefactors satisfying
\begin{equation}\label{uniform_initial_prefactor}
\limsup_{n\to\infty}D_0^n\leq
\begin{cases}
D_0^f,&\text{in Case (i)},\\
D_0^{f,g},&\text{in Case (ii)}.
\end{cases}
\end{equation}
The convex-hull construction gives uniform bounds on the initial
geometric quantities and ensures $D_0^n\leq D_0^f$ in Case (i)
and $D_0^n\leq D_0^{f,g}$ in Case (ii). Together with
$d_X^n(t)\leq d_X^n(0)+tD_0^n$, these bounds yield a uniform
bound on $\mathcal W_n(2\bar\tau)$; hence, (2.52) and (1.14)
allow $d^*$, and consequently
$\Lambda_*=\kappa(d^*)$ and
$\gamma=(3\bar\tau)^{-1}\log\bigl(1/(1-\Lambda_*\eta
e^{-4K\bar\tau})\bigr)>0$, to be chosen independently
of the population sizes.
\\
We recall the support property needed for the common limiting argument. If probability measures $\lambda_n\rightharpoonup\lambda$ narrowly, then
\begin{equation}\label{diameter_lsc}
\operatorname{diam}(\operatorname{supp}\lambda)
\leq\liminf_{n\to\infty}\operatorname{diam}(\operatorname{supp}\lambda_n).
\end{equation}
To prove this, take $a,b\in\operatorname{supp}\lambda$ and $\varepsilon>0$. The open balls $B(a,\varepsilon)$ and $B(b,\varepsilon)$ have positive $\lambda$-mass and, by the Portmanteau inequality, positive $\lambda_n$-mass for all sufficiently large $n$. Hence the diameter of $\operatorname{supp}\lambda_n$ is at least $|a-b|-2\varepsilon$. Taking the lower limit, then letting $\varepsilon\downarrow0$ and taking the supremum over $a,b$, proves \eqref{diameter_lsc}. \\
For every fixed $t\geq 0$, equation (4.76) gives
\[
\frac12\bigl(g_t^m+f_t^N\bigr)
\rightharpoonup
\frac12\bigl(g_t+f_t\bigr).
\]
The support of each mixture is the union of the supports
of its two components. Applying (4.79) to the velocity
marginals of these mixtures therefore gives
\[
d_V^f(t)\leq \liminf_{n\to\infty}d_V^n(t)
\quad\text{in Case (i)},
\]
and
\[
d_V^{f,g}(t)\leq \liminf_{n\to\infty}d_V^n(t)
\quad\text{in Case (ii)}.
\]
Thus, (4.77)-(4.78) yield
\[
\begin{aligned}
d_V^f(t)&\leq e^{-\gamma(t-2\bar\tau)}D_0^f
&&\text{in Case (i)},\\
d_V^{f,g}(t)&\leq e^{-\gamma(t-2\bar\tau)}D_0^{f,g}
&&\text{in Case (ii)}.
\end{aligned}
\]
No convergence of the individual support diameters is asserted. Since $T$ is arbitrary and these constants are independent of $T$, the bounds hold for all $t\geq0$. \\
Finally, let $(X,V)$ and $(\widetilde X,\widetilde V)$
be any two trajectories selected from either population.
In Case (i), these include both the follower characteristics
and the finitely many leader trajectories; in Case (ii),
they are characteristics of either kinetic population.
Since the position derivative equals the velocity along
every such trajectory,
\[
|X(t)-\widetilde X(t)|
\leq |X(0)-\widetilde X(0)|
+\int_0^t |V(s)-\widetilde V(s)|\,ds.
\]
Taking the supremum over all pairs of trajectories,
including pairs from different populations, and integrating
the corresponding velocity estimate, we obtain\[
\begin{aligned}
\sup_{t\geq0}d_X^f(t)
&\leq d_X^f(0)+\frac{e^{2\gamma\bar\tau}}{\gamma}D_0^f
&&\text{in Case (i)},\\
\sup_{t\geq0}d_X^{f,g}(t)
&\leq d_X^{f,g}(0)+\frac{e^{2\gamma\bar\tau}}{\gamma}D_0^{f,g}
&&\text{in Case (ii)}.
\end{aligned}
\]
This establishes the flocking conclusions under the uniform particle estimates above. The spatial bounds concern relative dispersion and do not require the absolute supports to remain in a fixed ball for all time.

\section{ Conclusions}
In this paper, we investigated a delayed leader-follower Cucker-Smale type system that incorporates a memory effect acting both in the leader and in the follower dynamics. \\
We first studied the asymptotic behaviour of the particle system and derived several qualitative properties of the dynamics. In particular, through a careful analysis of the evolution of the velocity diameter and a delayed Grönwall-type argument, we obtained uniform bounds on the spatial and velocity diameters of the system. These estimates allowed us to prove asymptotic flocking behavior under suitable assumptions on the interaction kernels and on the communication strength. \\
We then considered the corresponding mean-field limit, where the number of followers tends to infinity while the number of leaders remains fixed. This leads naturally to a hybrid ODE-PDE system consisting of a delayed Vlasov-type equation coupled with a finite-dimensional delayed leader dynamics. Using Wasserstein stability estimates together with the Lipschitz continuity of the delayed interaction field, we proved existence and uniqueness of global-in-time solutions. Using an adaptation of this method, we also proved the well-posedness of the fully mean-field regime obtained by sending the size of both populations to infinity.\\
The present work opens several possible research directions. An interesting problem is the study of control strategies and optimization mechanisms acting through the leaders, especially in the presence of delays and heterogeneous communication structures. It would also be relevant to investigate the emergence of more complex collective patterns, such as clustering or partial synchronization, as well as the effect of singular interaction kernels and state-dependent delays. 
Finally, the hybrid leader-follower framework developed here may provide a useful mathematical setting for the analysis of collective dynamics in applications involving autonomous agents, biological swarms, and networked multi-agent systems with memory effects.

	\bigskip
	\noindent {\bf Acknowledgements.} {\small The author gratefully acknowledges discussions with Cristina Pignotti. Moreover, the author sincerely thanks the anonymous referees for numerous suggestions and comments, which have improved the contents and the presentation of this paper significantly. }

\end{document}